\documentclass[12pt,fleqn]{amsart}

\usepackage[all]{xy}
\usepackage{amscd}
\usepackage{color}
\usepackage{enumerate}
\usepackage{amssymb}
\newtheorem{theorem}{Theorem}[section]
\newtheorem{lemma}[theorem]{Lemma}
\newtheorem{proposition}[theorem]{Proposition}
\newtheorem{problem}[theorem]{Problem}
\newtheorem{corollary}[theorem]{Corollary}

\theoremstyle{definition}
\newtheorem{definition}[theorem]{Definition}
\newtheorem{example}[theorem]{Example}

\theoremstyle{remark}
\newtheorem{remark}[theorem]{Remark}

\numberwithin{equation}{section}

\newcommand{\bA}{\mathbb A}
\newcommand{\cA}{\mathcal A}
\newcommand{\cB}{\mathcal B}
\newcommand{\cM}{\mathcal M}
\newcommand{\cN}{\mathcal N}

\newcommand{\cE}{\mathcal E}

\newcommand{\cI}{\mathcal I}

\newcommand{\cU}{\mathcal U}
\newcommand{\cV}{\mathcal V}

\newcommand{\bU}{\mathbb U}

\newcommand{\sub}{\subset}
\newcommand{\eps}{\varepsilon}
\newcommand{\er}{\mathbb R}
\newcommand{\sm}{\setminus}
\newcommand{\con}{\mathfrak c}

\newcommand{\wt}{\widetilde}
\newcommand{\vf}{\varphi}

\newcommand{\ext}{\protect{{\rm Ext}}}
\newcommand{\lra}{\longrightarrow}
\newcommand{\wh}{\widehat}

\newcommand{\sfrown}{\!\smallfrown\!}
\begin{document}
\baselineskip=17pt

\title[Twisted sums of $c_0$ and $C(K)$-spaces]{Twisted sums of $c_0$ and $C(K)$-spaces:\\ A solution to the CCKY problem}

\author[A.\ Avil\'es]{ Antonio Avil\'{e}s}
\address{Departamento de Matem\'{a}ticas\\
Facultad de Matem\'{a}ticas\\ Universidad de Murcia\\ 30100 Espinardo (Murcia)\\
Spain} \email{avileslo@um.es}

\author[W.\ Marciszewski]{Witold Marciszewski}
\address{Institute of Mathematics\\
University of Warsaw\\ Banacha 2\newline 02--097 Warszawa\\
Poland} \email{wmarcisz@mimuw.edu.pl}

\author[G.\ Plebanek]{Grzegorz Plebanek*}
\address{Instytut Matematyczny\\ Uniwersytet Wroc\l awski\\ Pl.\ Grunwaldzki 2/4\\
50-384 Wroc\-\l aw\\ Poland} \email{grzes@math.uni.wroc.pl}

\thanks{
The first author has been supported by project MTM2017-86182-P (AEI, Spain and ERDF/FEDER, EU).
The third author has been supported by the grant 2018/29/B/ST1/00223 from National Science Centre, Poland.}
\thanks{* corresponding author}

\begin{abstract}
We consider the class of Banach spaces $Y$ for which  $c_0$ admits a nontrivial twisted sum with $Y$.
We present a characterization of such spaces $Y$ in terms of properties of the $weak^\ast$ topology on $Y^\ast$. We prove that under the continuum hypothesis $c_0$ has a nontrivial twisted sum
with every space of the form $Y=C(K)$, where
$K$ is compact and not metrizable. This gives a consistent affirmative solution to a problem posed by Cabello, Castillo, Kalton and Yost.
\end{abstract}

\date{}
\maketitle

\section{Introduction}

A \emph{twisted sum} of Banach spaces $Y$ and $Z$ is a short exact sequence
\[ 0\lra Y\lra X \lra Z\lra 0, \]
where $X$ is a Banach space and the maps are bounded linear operators.
Such a twisted sum is called \emph{trivial} if the exact sequence splits, i.e.,
if the map $Y\lra X$ admits a left inverse (equivalently, if the map $X\lra Z$ admits a right inverse).
This is equivalent to saying that  the range of the map $Y\lra X$ is complemented in $X$; in this case, $X$ is isomorphic to $Y\oplus Z$.
We can, informally, say that $Y$ admits a nontrivial twisted sum with $Z$, if there exists a Banach space $X$ containing a uncomplemented copy $Y'$ of  $Y$  such that the quotient space  $X/Y'$ is isomorphic to $Z$.
The algebra of exact sequences and twisted sums is well-developed and found numerous applications in the Banach space theory, see
e.g.\ \cite{CCKY03,CC04,Ca01} and  recent monographs \cite{SIBS} and \cite{CC20}. Recall that $\ext(Z,Y)=0$ means that every twisted sum
of $Y$ and $Z$ is trivial.

By the classical Sobczyk theorem the space $c_0$ is separably injective, that is any isomorphic copy of $c_0$ is complemented in any separable superspace.
This implies $\ext(Y,c_0)=0$ for every separable Banach space $Y$. In particular, $\ext(C(K),c_0)=0$ whenever $K$ is compact and metrizable.
The following problem originated in \cite{CCKY03} and \cite{CCY00}; see also \cite{CC20} where it is called the CCKY problem.

\begin{problem}\label{i:1}
Given  a nonmetrizable compact space $K$, does there exist  a nontrivial twisted sum of $c_0$ and $C(K)$?
\end{problem}

There are several classes of nonmetrizable compacta for which  Problem \ref{i:1} has an affirmative answer, see  Castillo \cite{Ca16},
Correa and Tausk \cite{CT16}, Marciszewski and Plebanek \cite{MP18} and Correa \cite{Co18}.
However, it is proved in \cite{MP18} that if $\omega_1<\con$ and Martin's axiom holds then
$\ext(C(2^{\omega_1}), c_0)=0$ and  $\ext(C(K), c_0)=0$ for separable scattered compacta of height $3$ and weight $\omega_1$. The latter result has been recently generalized by Correa and Tausk \cite{CT18}. Hence, Problem \ref{i:1} has a consistent negative answer.

In this paper we present a number of theorems related to Problem \ref{i:1} and, more generally, to the class of Banach spaces $Y$
for which $\ext(Y,c_0)\neq 0$. Here is the list of our main results.

\begin{description}

\item[Theorem \ref{main}] We present  a characterization of Banach spaces $Y$ with $\ext(Y, c_0)= 0$ in terms of some properties of
the $weak^\ast$ topology on $Y^\ast$. The result has a variety of consequences, usually of a cardinal nature, enabling one to
 check that $\ext(Y, c_0)\neq 0$ for several types of  $Y$, see Section \ref{when} for details.
\item[Theorem \ref{ch:8}] Assuming the continuum hypothesis (CH),  $\ext(C(K), c_0)\neq 0$ for every nonmetrizable compact space $K$.
This is   a consistent affirmative solution to Problem \ref{i:1}.
The result  is a consequence of a sequence of auxiliary results in Section \ref{ch} based on Theorem \ref{main}.
\item[Theorem \ref{scattered:1}] We prove, without additional axioms,  that $\ext(C(K), c_0)\neq 0$ for every scattered compactum $K$ of finite
height and weight $\ge \con$ (in fact, of weight $\ge \mathrm{cf}(\con))$. Such a result was first obtained by
Castillo \cite{Ca16} under CH  and Correa \cite{Co18} under Martin's axiom.

\item[Section \ref{tsad}] The section offers a proof that $\ext(Y,c_0)\neq 0$ whenever $Y$ is a Banach space whose dual unit ball
contains suitably placed copies, in the weak$^\ast$ topology, of the Aleksandrov compactification of discrete sets.
This, in particular,  implies that $\ext(C(2^\kappa),c_0)\neq 0$ for cardinals $\kappa$ that are consistently  smaller than $\con$,
and $\ext(C(K),c_0)\neq 0$ whenever $K$ is a compact space that can be continuously mapped onto $[0,1]^\con$.
\end{description}

The basic idea behind Theorem \ref{main} is that twisted sums of $c_0$ and $Y$ are tightly connected with properties of compact spaces
that may be constructed by adding  a countable set of isolated points to  the dual unit ball in $Y^\ast$.

In the context of $C(K)$ spaces such an  approach was already
used in \cite{MP18}. The characterization given by \ref{main} connects the question on $\ext(Y,c_0)$ with a number of subtle properties of compact spaces
and some set-theoretic considerations.

The theorems presented in  of Section \ref{tsad} extend some results due to Correa and Tausk \cite{CT16} and use auxiliary results on almost disjoint families
of subsets of $\omega$ that are discussed in Section \ref{ad}. It turns out that for some compacta $K$ the question whether
$\ext(C(K),c_0)\neq 0$ has a highly set-theoretic nature. For example, $\ext(C(2^\kappa),c_0)\neq 0$ whenever there is a subset
of the real line of cardinality $\kappa$ having full outer measure.

It has become clear that Problem \ref{i:1} is undecidable in the usual set theory; however, the following remains open.

\begin{problem}\label{i:2}
Is there a consistent example of a compact space $K$ of weight $\ge\con$ for which $\ext(c_0,C(K))= 0$?
\end{problem}

One might, for example, consider  Problem \ref{i:2} for scattered compacta of countable height, such as the space described in section \ref{ap1}.

\section{Preliminaries}\label{prel}

If $K$ is a compact space then $C(K)$ is the familiar Banach space of continuous real-valued functions on $K$.
We usually identify $C(K)^*$ with the space $M(K)$ of signed Radon measures on $K$ of finite variation. $M_1(K)$ stands for the closed unit ball in
$M(K)$; given $r>0$, we denote $r\cdot M_1(K)$ by $M_r(K)$. Most often  $M_r(K)$ is  equipped with the $weak^\ast$ topology inherited from $C(K)^\ast$. Every
signed measure
$\mu\in M(K)$ can be written as  $\mu=\mu^+ -\mu^-$, where $\mu^+,\mu^-$ are nonnegative mutually singular measures.

We write  $\mu(f)$ for $\int_K f\;{\rm d}\mu$, unless some operations on the integral are needed (as in the standard lemma below).

\begin{lemma}\label{prel:-1}
If $V\sub K$ is an open set then the set of the form
\[U=\{\mu\in M(K): \mu^+(V)>a\},\]
is $weak^\ast$ open in $M(K)$.
\end{lemma}

\begin{proof}
	By the definition of the weak$^\ast$ topology, each set of the form $W_f=\{\mu\in M(K) : \mu(f)>a\}$ is weak$^\ast$-open for any $f\in C(K)$. Therefore it is enough to check that
\[ U = \bigcup\left\{W_f : f\in C(K), f:K\longrightarrow [0,1] \text{ and } f|_{K\setminus V} = 0\right\}.\]
	Let $V = V^+ \cup V^-$ be a decomposition of $V$ into two Borel sets where the positive and negative part of $\mu$ are concentrated, respectively.
 For the inclusion $[\supseteq]$, if $\mu\in W_f$ for some $f$ as above, then
 \[a<\mu(f) = \int_{V} f\;{\rm d}\mu= \int_{V^+} f\;{\rm d}\mu + \int_{V^-} f\;{\rm d}\mu
    \le \int_{V^+} f\;{\rm d}\mu\le  \mu^+(V).\]
	For the reverse inclusion $[\subseteq]$ fix $\mu\in U$. By the regularity of the measure we can find closed sets $F^+\subset V^+$ and $F^-\subset V^-$ such that $|\mu|(V^\sigma\setminus F^\sigma)<(\mu^+(V)-a)/4$ for $\sigma\in\{+,-\}$. Take  a continuous function $f:K\longrightarrow [0,1]$ such that $f|_{F^+} = 1$, $f|_{F^-}=0$ and $f|_{K\setminus V}=0$. Then $\mu\in W_f$ because
	\[ \mu(f)\ge \int_{F^+} f\;{\rm d}\mu + \int_{F^-} f\;{\rm d}\mu
 -(\mu^+(V)-a)/2 > \mu({V^+})-\frac{3}{4}(\mu^+(V)-a)>a.\]
\end{proof}

The symbol $P(K)$ denotes the subspace of $M_1(K)$ consisting of all probability measures; given $x\in K$, $\delta_x\in P(K)$ is the Dirac measure, a point mass concentrated at the point $x$. Note that $K$ is homeomorphic to the subspace $\Delta_K=\{\delta_x: x\in K\}$ of  $P(K)$.

\subsection{Operators}
We recall the following standard fact for the future reference.

\begin{lemma}\label{prel:0}
Let $Y_0$ be a complemented subspace of a Banach space $Y$.
If $\ext(Y_0,c_0)\neq0$ then $\ext(Y,c_0)\neq 0$.
In particular, if $L\sub K$ is a retract of a compact space $K$ and $\ext(C(L),c_0)\neq 0$ then $\ext(C(K), c_0)\neq0$.
\end{lemma}

\begin{proof}
The first assertion is easy, because if 
$0\lra c_0 \lra X \lra Y_0 \lra 0$ is a nontrivial short exact sequence, then the sequence
 $ 0 \lra c_0 \lra X\oplus A \lra Y_0 \oplus A\lra 0$ is also nontrivial.


The second assertion follows from the first one and the fact that if $r:K\lra L$ is a retraction then $C(L)$ is isometric to a complemented subspace
$X=\{g\circ r: g\in C(L)\}$ of $C(K)$, where the projection $P:C(K)\lra X$ is given by $Pf=(f|L)\circ r$ for $f\in C(K)$.
\end{proof}

If $T:X\lra Y$ is a bounded linear operator between Banach spaces then $T^\ast:Y^\ast\lra X^\ast$ is the conjugate operator
given by $T^\ast y^\ast(x)=y^\ast(Tx)$. Recall that if $T$ is surjective then $T^*[Y^*]=\ker(T)^\perp $.


\subsection{Countable discrete extensions and compactifications of $\omega$}
If $K$ is a compact space then we call a compact space $L\supset K$  a {\em countable discrete extension} of $K$ if
$L\sm K$ is a countable infinite discrete set. Whenever possible, we identify $L\sm K$ with $\omega$.
Note that if $\gamma\omega$ is some compactification of $\omega$ then $\gamma\omega$ is a countable discrete extension
of its remainder $\gamma\omega\!\sm\!\omega$.

We write $\omega^\ast$ for the remainder of the \v{C}ech-Stone compactification $\beta\omega$ of natural numbers.
We shall frequently use  the following  fact, see  \cite[3.5.13]{En}.

\begin{lemma}\label{prel:2}
If $f:\omega^\ast\lra K$ is a continuous surjection then $K$ is  homeomorphic to the remainder of a compactification
$\gamma\omega$ and $f$ can be extended to a continuous function $\beta\omega\lra\gamma\omega$.
\end{lemma}

The space $\omega^\ast$ is projective with respect to metrizable compacta, i.e.\ the following holds, see \cite[Corollary 5.24]{SIBS}.

\begin{theorem}\label{prel:2.5}
If $K_1,K_2$ are metrizable compacta and $g:K_1\lra K_2$ is a continuous surjection then for every continuous map $\vf:\omega^\ast\lra K_2$
there is a continuous map $\psi:\omega^\ast\lra K_1$ such that the following diagram commutes
$$\xymatrix{K_1\ar[rr]^g&&K_2\\
\omega^\ast\ar@{.>}_{\psi}[u]  \ar_{\vf}[urr]&& }$$
\end{theorem}

If $(x_n)$ is a sequence in a compact space $K$ and $\cU$ is a non-principal ultrafilter on $\omega$ then
$x=\lim_{n\to\cU} x_n$ is the unique point in $K$ such that $\{n\in\omega: x_n\in U\}\in\cU$ for  every open neighborhood $U$ of $x$.
Recall that every accumulation point of $x_n$'s can be written as $ \lim_{n\to\cU} x_n$ for  some ultrafilter $\cU$.

\subsection{Almost disjoint families and Aleksandrov-Urysohn compacta} \label{prel:adf}
A topological space $X$ is \emph{scattered} if no nonempty subset $A\subseteq X$ is dense-in-itself.
For an ordinal $\alpha$, $X^{(\alpha)}$ denotes the $\alpha$th
Cantor-Bendixson derivative of the space $X$. For a scattered
space $X$, the scattered height of $X$ is
${ht(X)}=\min\{\alpha: X^{(\alpha)} =
\emptyset\}$.

We write $\bA(\kappa)$ for the (Aleksandrov) one-point compactification
of a discrete space of cardinality $\kappa$.

Recall that a family $\cA$ of infinite subsets of $\omega$ is almost disjoint if $A\cap B$ is finite for any distinct $A,B\in \cA$.
To every almost disjoint family $\cA$ one can associate an Aleksandrov-Urysohn  compactum $\bA\bU(\cA)$ of height 3.
That space may be simply defined
as the Stone space of the algebra of subsets of $\omega$ generated by $\cA$ and all finite sets.
In other words,
\[\bA\bU(\cA)=\omega\cup \{A:A\in\cA\}\cup\{\infty\},\]
where points in $\omega$ are isolated, basic open neighborhoods of a given point $A$ are of the form $\{A\}\cup (A\sm F)$ with $F\sub \omega$ finite, and $\bA\bU(\cA)$ is the one point compactification of
the locally compact space $\omega\cup \{A:A\in\cA\}$, where $\infty$ is the point at
infinity.

Recall that combinatorial properties of $\cA$ are often reflected by topological properties of $\bA\bU(\cA)$, see
 the survey paper  Hru\v{s}\'ak \cite{Hr14}. For instance, if
$\cA$ is a maximal almost disjoint family then $\infty$ lies in the closure of ${\omega}$ but no subsequence of $\omega$ converges to $\infty$
(so the resulting space is not Fr\'echet-Urysohn).

The class of spaces $\bA\bU(\cA)$ is usually associated with the names of Mr\'owka, Isbell, or Franklin; however,
 such compacta were already considered by Aleksandrov and Urysohn \cite{AU29} and there seems to be a good reason
to call them  Aleksandrov-Urysohn compacta, cf.\ \cite{MP09}.

\subsection{Parovi\v{c}enko's theorem}
Parovi\v{c}enko's theorem states that every compact space of weight $\le \omega_1$ is a continuous image of the space $\omega^\ast$, see
\cite{Pa63}. This implies that every Banach space of density $\le\omega_1$ can be isometrically embedded into $C(\omega^\ast)$ which is an isometric copy of
the classical Banach space $\ell_\infty/c_0$. In particular, under CH the space $C(\omega^\ast)$ is universal for the class of Banach spaces of density
not exceeding $\con$; Koszmider \cite{Ko15} offers a detailed discussion on the existence of universal object in several classes of Banach
spaces.

\subsection{Some classes of compacta}

Let us recall that a compact space $K$ is {\em Eberlein compact} if $K$ is homeomorphic to a weakly compact subset of a Banach space.
There are well-studied much wider classes of Corson and Valdivia compacta. Given a cardinal number $\kappa$,  the $\Sigma$-product $\Sigma(\er^\kappa)$ of real lines is  the subspace of
$\mathbb{R}^\kappa$ consisting of functions with countable support.
A compactum $K$ is  {\em Corson compact} if it can be
embedded into some $\Sigma(\er^\kappa)$;
$K$ is {\em Valdivia compact} if for some $\kappa$ there is an embedding $g:K\lra\er^\kappa$ such that
$g(K)\cap\Sigma(\er^\kappa)$ is dense in the image,
see Negrepontis \cite{Ne84}, Argyros, Mercouraks and Negrepontis \cite{AMN} and Kalenda \cite{Ka00}.

Denote by
$B_1(\omega^\omega)$ the space of first  Baire class functions $\omega^\omega\lra \er$, equipped with the pointwise topology. A compact space $K$ is
said to be {\em Rosenthal compact} if $K$ can be topologically embedded into $B_1(\omega^\omega)$, see \cite{Mar03} for basic properties
of Rosenthal compacta and further references.

\section{Twisted sums and discrete extensions}\label{ts}

If $K$ is a compact space, $L$ is a countable discrete extension of $K$ and $Z$ is another topological superspace of $K$, we say that $L$ {\em can be realized inside} $Z$, if the inclusion map $K\longrightarrow Z$ extends to a homeomorphic embedding $L\longrightarrow Z$. If $T:Y\lra \ell_\infty/c_0$ is a bounded operator then we say that $T$ can be lifted to $\ell_\infty$ if there is a bounded operator
$\wt{T}:Y\lra\ell_\infty$ closing the  following diagram
$$\xymatrix{
	Y\ar[rr]^T \ar@{.>}_{\wt{T}}[d]&&\ell_\infty/c_0\\
	\ell_\infty  \ar_{Q}[urr]&& }$$
where $Q:\ell_\infty\lra \ell_\infty/c_0$ is the quotient operator.

The main results of this section are Theorem~\ref{main} and~\ref{main2} below. The equivalence of $(i)$ and $(v)$ in Theorem~\ref{main} is known, one can find a proof of it using homological tools in
\cite[Proposition 1.4.f] {CG97}. The other conditions $(ii)$, $(iii)$ and $(iv)$ can be viewed as topological counterparts of $(v)$.

\begin{theorem}\label{main}
For an infinite dimensional Banach space $Y$ the following are equivalent:
\begin{enumerate}[(i)]
	\item $\ext(Y,c_0) = 0$;
	\item every countable discrete extension of $(B_{Y^\ast},weak^\ast)$ can be realized inside  $(Y^\ast,weak^\ast)$;
	\item every countable discrete extension of any compact subset of $(Y^\ast,weak^\ast)$ can be realized inside $(Y^\ast,weak^\ast)$;
	\item every continuous function $\omega^\ast\longrightarrow (Y^\ast,weak ^\ast)$ extends to a continuous function
$\beta\omega\longrightarrow (Y^\ast,weak^\ast)$;
	\item every bounded operator $T:Y\lra \ell_\infty/c_0$ can be lifted to $\ell_\infty$.
\end{enumerate}	
\end{theorem}

\begin{proof}
We start with the following general remark:
\medskip

{\sc Remark.} Let $(y^\ast_n)$ be a bounded sequence in $Y^\ast$. Then, for any $a>0$, there is a sequence $(x^\ast_n)$ of pairwise distinct points of $Y^\ast$, lying outside $a\cdot B_{Y^\ast}$, and such that  $x_n^\ast-y_n^\ast$ converge to $0$ in the  $weak^\ast$ topology.
\medskip

Indeed, let $\|y^\ast_n\|\le c$ for all $n$.
By the Josefson-Nissenzweig theorem, we can pick a $weak^\ast$ null sequence $(z_n^\ast)$ in $Y^\ast$ such that $\|z_n^\ast\| = c+a+1$ for all $n$, see
Diestel \cite[Chapter 12]{Di84}.  Then we may choose vectors $u_n^\ast$ with $1>\|u_n^\ast\|\lra 0$, so that  the vectors
\[ x_n^\ast=y^\ast_n+z_n^\ast+u_n^\ast,\]
are pairwise distinct.
Then it is clear that the sequence $(x^\ast_n)$ has the required properties. Note that we also have the estimate $\|x^\ast_n\|\le 2c+a+2$.
\medskip


The implication   $(iii) \lra (ii)$ is obvious; for  the reverse implication
take $K\sub B_{Y^\ast}$ and its countable discrete extension $K\cup \omega$ (we think that $\omega$ is disjoint from $Y^\ast$).
Then we consider $B_{Y^\ast}\cup\omega$ to conclude the proof.
Hence, $(ii)$ and $(iii)$ are equivalent.

To prove $(i)\to (ii)$
	suppose that $\ext(Y,c_0)=0$, and let $L = B_{Y^\ast}\cup \omega$ be a countable discrete extension of the dual unit ball. We consider the following subspace of the Banach space of continuous functions $C(L)$:
	\[ X = \left\{x\in C(L) : (\exists y\in Y) \  (\forall y^\ast\in B_{Y^\ast}) \ x(y^\ast)=y^\ast(y) \right\}.\]
	We have a short exact sequence
	\[0\longrightarrow c_0  \overset{u}{\lra} X \overset{p} \longrightarrow Y \longrightarrow 0,\]
	where the  operator $u: c_0\lra X$ sends an element $x\in c_0$ to the continuous function $u(x)$ on $L$ that acts like $x$ on $\omega$ and vanishes on $B_{Y^\ast}$, while
 the operator $p: X\lra Y$ sends a function $x\in X$ to the unique $y=p(x)$ such that $x(y^\ast) = y^\ast(y)$ for all $y^\ast\in B_{Y^\ast}$. We know that this exact sequence splits, so there is an operator $E:Y\longrightarrow X$ such that $E(y)(y^\ast) = y^\ast(y)$ for all $y^\ast\in B_{Y^\ast}$. For every $n$  define $y_n^\ast = E^\ast(\delta_n)$, where $\delta_n\in X^\ast$ is the evaluation at the point $n\in\omega\sub L$. A natural candidate for a realization of $L$ inside $Y^\ast$ is $B_{Y^\ast}\cup \{y_n^\ast: n\in\omega\}$ (which follows directly from the formula $y^\ast_n(y)=Ey(n)$).
  Note that we may have that $y_n^\ast\in B_{Y^\ast}$; however using {\sc Remark} (for $a=1$), we can replace $y_n^\ast$ by $x_n^\ast$ and conclude that $L$ can be realized inside $Y^\ast$.
	
We shall now prove that $(ii)\to (i)$.
	Consider  a short exact sequence 	
\[ 0\longrightarrow c_0 \overset{u}{\longrightarrow} X \overset{p}{\longrightarrow} Y \longrightarrow 0.\]
To check that it splits we examine the dual sequence
	\[ 0\longleftarrow c_0^\ast \overset{u^\ast}{\longleftarrow} X^\ast \overset{p^\ast}{\longleftarrow} Y^\ast \longleftarrow 0.\]
	Let $\{e_n^\ast : n<\omega\}$ be the dual unit basis in $c_0^\ast=\ell_1$. By the open mapping theorem, we find a bounded sequence $\{x_n^\ast\}\subset X^\ast$ such that $u^\ast(x_n^\ast) = e_n^\ast$ for all $n$.

Note that if $x^\ast$ is a $weak^\ast$ cluster point of $x_n^\ast$'s then from the equality $u^\ast(x_n^\ast) = e_n^\ast$ it follows that $x^\ast$
vanishes on $u(c_0)$. We conclude  from  the fact that $u(c_0)$ is the kernel of $p$, that $x^\ast$ lies in $p^\ast(Y^\ast)$, but $x_n^\ast\not\in p^\ast(Y^\ast)$ for any $n$.
It follows that  for some radius $r>0$, we have a countable discrete extension $L=p^\ast(r\cdot B_{Y^\ast}) \cup \{x_n^\ast : n<\omega\}$
of the $weak^\ast$ compact set $p^\ast(r\cdot B_{Y^\ast})$ which is a homeomorphic copy of $K=r\cdot B_{Y^\ast}$. Identifying $p^\ast(r\cdot B_{Y^\ast})$ with $r\cdot B_{Y^\ast}$ (via $p^\ast$) we can treat $L$ as a countable discrete extension of $K$.
By our assumption,  we can realize this extension as a bounded set $r\cdot B_{Y^\ast}\cup \{y_n^\ast\}\subset Y^\ast$.

In order to prove, that the above exact sequence splits, we shall find an operator $T:X\longrightarrow c_0$ such that $Tu$ is the identity on $c_0$.
We define such $T$ by the formula
\[Tx = (x_n^\ast(x) - y_n^\ast(px))_{n<\omega},\]
 and check that it is as required.

 First  we prove that $Tx$ indeed belongs to $c_0$. Since the functionals $y_n^\ast$ were chosen to realize the  countable discrete extension given by
 $x_n^\ast$, for every nonprincipal ultrafilter $\mathcal{U}$ on $\omega$ we have
 $\lim_{n\to \mathcal{ U}}x_n^\ast = p^\ast(\lim_{n\to \mathcal{U}} y_n^\ast),$ in the $weak^\ast$ topology.
 Thus
 \[ \lim_{n\to \mathcal{U}}x_n^\ast(x) = p^\ast(\lim_{n\to \mathcal{U}} y_n^\ast)(x) =
  \lim_{n\to \mathcal{U}} p^\ast(y_n^\ast)(x) = \lim_{n\to\mathcal{U}} y_n^\ast(px),\]
  for every $x\in X$,
 and this shows that $Tx \in c_0$. Since the sequences $(x_n^\ast)$ and $(y_n^\ast)$ are both bounded, $T$ is a bounded linear operator. Finally, we prove that $Tu$ is the identity on $c_0$.  Take $z=(z_n)\in c_0$;  since $pu=0$ and $u^\ast(x_n^\ast)=e_n^\ast$, then $n$-th coordinate of $Tuz$ is
	\[ x_n^\ast(uz) - y_n^\ast(puz) = x_n^\ast(uz) = u^\ast(x_n^\ast)(z) = e_n^\ast(z) = z_n.\]

$(iii)\to (iv)$. 	Let $f:\omega^\ast\lra Y^\ast$ be a continuous mapping. Then, by Lemma \ref{prel:2}, the space $K=f(\omega^\ast)$ is a remainder of some compactification
$\gamma\omega$ and $f$ extends to   a continuous surjection  $g: \beta\omega\lra \gamma\omega$ that does not move natural numbers.
By our assumption, the inclusion map $\imath:K\lra Y^\ast$ extends into an embedding $\widetilde{\imath}:\gamma\omega\lra Y^\ast$, and
$\widetilde{\imath}\circ g:\beta\omega\lra Y^\ast$ is the required extension of $f$.

$(iv)\to (iii)$.
Take a $weak^\ast$ compact set $K\sub Y^\ast$ and consider its countable discrete extension
$L=K \cup\omega$. Let $K_0$ be the closure of $\omega$ in $L$. Then $K_0$ is a compactification of $\omega$
so there is a continuous map $F:\beta\omega \lra K_0$ which is the identity on $\omega$ and $f= F|\omega^\ast$ is a surjection of $\omega^\ast$ onto $K_0\setminus\omega$. By $(iv)$ there is an extension of $f$ to a continuous
mapping $\widehat{f}:\beta\omega\lra Y^*$. Using {\sc Remark} from the beginning of the proof,
we can assume that $\widehat{f}(n)\notin K$ for every $n$ and
$\widehat{f}(n)\neq \widehat{f}(k)$ for $n\neq k$. Then the inclusion map $K\lra Y^\ast$ extends to an embedding $K\cup\omega\lra Y^\ast$
sending $n\in\omega$ to $\widehat{f}(n)$.

Note that $(v)$ is equivalent to saying that every  bounded operator $T:Y\lra C(\omega^\ast)$ can be written as $T=R\circ\wt{T}$ for some
bounded operator $\wt{T}:Y\lra C(\beta\omega)$, where $R:C(\beta\omega)\lra C(\omega^\ast)$ is the restriction.

$(iv)\to (v)$.
Take   $T:Y\lra C(\omega^\ast)$ and consider the conjugate operator $T^\ast: M(\omega^\ast)\lra Y^\ast$.
Define $g:\omega^\ast\lra Y^\ast$ by $g(\cU)=T^\ast\delta_\cU$ for $\cU\in \omega^\ast$. Then $g$ is $weak^\ast$ continuous so by $(iv)$ it can be extended to
a continuous map $\wt{g}:\beta\omega\lra Y^\ast$. Let $y_n^\ast=\wt{g}(n)$ for every $n\in\omega$. Then we may defined the required
operator $\wt{T}:Y\lra C(\beta\omega)$ putting $\wt{T}y(n)=y^\ast_n$ for $n\in\omega$
and $\wt{T}y(\cU)=\lim_{n\to\cU} y_n^\ast(y)$ for $\cU\in\omega^\ast$, where $y\in Y$. Indeed, for any $\cU\in\omega^\ast$ and $y\in Y$,
\[ \wt{T}y(\cU)=\lim_{n\to\cU} y_n^\ast(y)=\lim_{n\to\cU} \wt{g}(n)(y)=g(\cU)(y)=T^\ast \delta_\cU (y)=Ty(\cU).\]

$(v)\to (iv)$.
Consider a continuous map $g:\omega^\ast\lra Y^\ast$.
We can define an operator $T:Y\lra C(\omega^\ast)$ by $Ty(\cU)=g(\cU)(y)$ and lift it to $\wt{T}:Y\lra C(\beta\omega)$. Then it is easy to check that
putting $\wt{g}(n)=\wt{T}^\ast\delta_n$ we define a continuous extension $\wt{g}:\beta\omega\lra Y^\ast$ of $g$.
\end{proof}

It will be useful in the sequel to have the following `bounded' version of Theorem \ref{main}.

\begin{theorem}\label{main2}
For an infinite dimensional Banach space Banach space $Y$ the following are equivalent:
\begin{enumerate}[(i)]
	\item $\ext(Y,c_0) = 0$;
	\item there is a constant $r>0$ such that  every countable discrete extension of $(B_{Y^\ast},weak^\ast)$ can be realized inside
$(r\cdot B_{Y^\ast},weak^\ast)$;
	\item there is a constant $r>0$ such that every continuous function $\omega^\ast\longrightarrow (B_{Y^\ast},weak ^\ast)$ extends to a continuous function
$\beta\omega\longrightarrow (r\cdot B_{Y^\ast},weak^\ast)$;
\item there is a constant $r>0$ such that
every bounded operator $T:Y\lra \ell_\infty/c_0$ can be lifted to $\tilde{T}:Y\lra \ell_\infty$ with $\|\tilde{T}\|\leq r \|T\|$.
\end{enumerate}	
\end{theorem}

\begin{proof}
The implication $(ii)\to (i)$ follows directly from Theorem \ref{main}.

To justify $(iii)\to (ii)$ we can repeat the argument from the proof of $(iv)\to (iii)$ in Theorem \ref{main}. Here we should also use the estimate of the norm given and the end of the proof of {\sc Remark}. This estimate allows to show that if a constant $r$ satisfies the condition from $(iii)$, then the constant $2r+3$ satisfies the condition from $(ii)$ (we use {\sc Remark} for $a=1$ and $c=r$). In a similar, even easier way, the proof of the equivalence $(iv)\Leftrightarrow (v)$ in Theorem~\ref{main} gives now $(iii)\Leftrightarrow (iv)$.

It remains to check $(i)\to (iii)$.

Suppose that $(iii)$ does not hold; then for every $n$ there is a continuous map $\vf_n:\omega^\ast\lra B_{Y^\ast}$ that cannot be
extended to a continuous function $\beta\omega\lra n\cdot B_{Y^\ast}$. We consider the remainder $(\omega\times\omega)^\ast$
of $ \beta(\omega\times\omega)$,  the \v{C}ech-Stone compactification of $\omega\times\omega$,
which is clearly homeomorphic to  $\omega^\ast$. Define a function $\vf:(\omega\times\omega)^\ast\lra Y^\ast$ by the formula
\[
\vf(\cU)=
\begin{cases}
    \frac{1}{\sqrt{n}}\vf_n(\cU^n), & \text{if } \omega\times\{n\}\in\cU, \\
    0 & \text{if } \omega\times\{n\}\notin\cU \text{ for every } n,
  \end{cases}
\]
for  any nonprincipial ultrafilter $\cU$ on $\omega\times\omega$; here
 $\cU^n$ denotes
 is the ultrafilter on $\omega$ defined by $A\in \cU^n$ iff $A\times\{n\}\in\cU$ (in case when  $\omega\times\{n\}\in\cU$).
It is clear that $\vf$ is continuous when we put on $Y^\ast$ the $weak^\ast$ topology.

Let us check that  $\vf$ does not have a continuous extension $\widetilde{\vf}:\beta(\omega\times\omega)\lra Y^\ast$. Otherwise,
for every $n$,  the function  $\sqrt{n}\, \widetilde{\vf}$ restricted to $\omega\times\{n\}$ gives a continuous extension  $\widetilde{\vf_n}: \beta\omega\lra Y^\ast$ of $\vf_n$.
By the choice of $\vf_n$, we have $\widetilde{\vf_n}(\beta\omega)\not\subset n\cdot B_{Y^\ast}$; consequently,
 $\widetilde{\vf}(\beta(\omega\times\omega))\not\subset\sqrt{n}\cdot B_{Y^\ast}$ for every $n$,
contrary to the fact that the image of $\widetilde{\vf}$ should be bounded.

Now, by the implication $(iv)\to (i)$ of  Theorem \ref{main}, we infer that  $\ext(Y,c_0)\neq 0$, and we are done.

\end{proof}


\section{Consequences of $\ext(Y,c_0)=0$}\label{when}

In this section we apply Theorem \ref{main} to show that the assumption $\ext(Y,c_0)=0$ has a strong impact on the properties of the $weak^\ast$ topology of $Y^\ast$.
This yields simple cardinal tests for Banach spaces admitting nontrivial twisted sums with $c_0$.
\newcommand{\dens}{\protect{\rm dens}}

Recall that a compact topological space $K$ is {\em monolithic} if and only if the density $\dens(F)$ coincides with the weight $w(F)$ for every closed subspace $F$ of $K$.

\begin{corollary}\label{when:1}
Let $Y$ be a Banach space satisfying $\ext(Y,c_0)=0$.

\begin{enumerate}[(a)]
\item If $Y$ is isomorphic to a subspace of $\ell_\infty/c_0$, then $Y$ is isomorphic to a subspace of $\ell_\infty$.
\item
If $\dens(Y)\leq \omega_1$, then $Y$ is isomorphic to a subspace of $\ell_\infty$.
\item
Every compact subset of $(Y^\ast, weak^\ast) $ of weight $\omega_1$ is contained in a $weak^\ast$--separable bounded subset of $Y^\ast$.
\item If $Y$ is nonseparable then  $(B_{Y^\ast},w^\ast)$ is not monolithic.
\end{enumerate}
\end{corollary}

\begin{proof}
Statement $(a)$ follows form Theorem~\ref{main} using condition $(v)$.

To check $(b)$ note that if $\dens(Y)\le\omega_1$ then $(B_{Y^\ast}, weak^\ast)$ is of topological weight $\le\omega_1$ and, by Parovi\v{c}enko's theorem,
is a continuous image of $\omega^\ast$. This means that $Y$ embeds into $C(\omega^\ast)$ and we may apply $(a)$.

Part $(c)$ follows from Parovi\v{c}enko's theorem and $(iv)\to (i)$ of Theorem \ref{main}.

We argue for $(d)$ as follows: Suppose that $(B_{Y^\ast},w^\ast)$ is monolithic. Then the $weak^\ast$ closure of any countable subset of $B_{Y^\ast}$ is metrizable, hence it is a proper subset of $B_{Y^\ast}$, which is nonmetrizable. Therefore, we can choose inductively  $y^\ast_\alpha\in B_{Y^\ast}$, for $\alpha<\omega_1$, such that $y^\ast_\alpha$ does not belong to the $weak^\ast$ closure of $\{y_\beta^\ast:\beta<\alpha\}$. Clearly, this implies that the set $A=\{y_\alpha^\ast:\alpha<\omega_1\}$ is not separable. Let $F$ be the $weak^\ast$ closure of $A$. By monolithicity $w(F)\le\omega_1$ and, by $(c)$, $F$ is contained in a $weak^\ast$ separable bounded set $F_1$. Then, again by monolithicity,  $F_1$ and $A$ are separable metrizable, a contradiction.
\end{proof}

Recall that every Corson compact space is monolithic. Hence Corollary \ref{when:1}(d) extends the result stating that
$\ext(Y,c_0)\neq 0$ for every nonseparable
Banach space $Y$ which is weakly Lindel\"of determined, i.e.\ such that $B_{Y^\ast}$ is Corson compact in its $weak^\ast$ topology.
It is not true in ZFC that $(B_{C(K)^\ast},weak^\ast)$ is Corson when $K$ is Corson, but we can say the following (cf. \cite[Theorem 10.2]{MP18}):

\begin{corollary}\label{some:2}
	Under Martin's axiom, $\ext(C(K),c_0)\neq 0$ for every nonmetrizable Corson compact space $K$.
\end{corollary}

\begin{proof}
	The above result was proved under CH by Correa and Tausk \cite{CT16}, and it will also come as a corollary of our Theorem \ref{ch:8}.	Under MA + $\neg$CH,  the dual unit ball $L$ in $C(K)^\ast$ is Corson compact in  its $weak^\ast$ topology for every Corson $K$, see \cite{AMN}.
	In particular,  $L$ is monolithic so $\ext(C(K),c_0)\neq 0$ by Corollary  \ref{when:1}{\em (d)}.
\end{proof}

Below we collect several cardinal restrictions on the $weak^\ast$ topology in a dual of the space $Y$ satisfying $\ext(Y,c_0)=0$.
Here $C(T,S)$ stands for the set of all continuous functions between topological spaces $T$ and $S$.

\begin{corollary}\label{when:2}
Let $Y$ be a Banach space satisfying $\ext(Y,c_0)=0$.

\begin{enumerate}[(a)]
\item \label{1}
If $K$ is a weak$^\ast$ compact subset of $Y^\ast$, then the number of (pairwise non-homeomorphic) countable discrete extensions of $K$ is bounded by $|Y^\ast|$.
\item   \label{2}
 $|C(\omega^\ast,B_{Y^\ast})|\leq |Y^\ast|$.
\item \label{5}
If $(Y^\ast, weak^\ast) $ contains a copy of $\mathbb{A}(\kappa)$ with $\kappa\leq\mathfrak{c}$, then $|Y^\ast| \ge  2^\kappa$.
 \end{enumerate}
\end{corollary}

\begin{proof}
Recall first that $|X|^\omega=|X|$ for every Banach space $X$, see \cite[p. 184]{Co71}; in particular, we have $|Y^\ast|^\omega=|Y^\ast|$.
Thus {\em (\ref{1})} follows directly from Theorem \ref{main}.

For {\em (\ref{2})} we use
$(iv)$ of Theorem \ref{main}  and the fact that $|C(\beta\omega, Y^\ast)|\le |Y^\ast|^\omega=|Y^\ast|$.

For {\em (\ref{5})} observe that if $\kappa\leq\mathfrak{c}$ then $\bA(\kappa)$ contains $2^\kappa$ many closed subsets and each of them
is a continuous image of $\omega^\ast$. In particular, $|C(\omega^\ast, \bA(\kappa))|\ge 2^\kappa$ and
if $\bA(\kappa)$ embeds into $B_{Y^\ast}$ then $|C(\omega^\ast, B_{Y^\ast})|\ge 2^\kappa$ so $|Y^\ast|\ge 2^\kappa$ by {\em (\ref{2})}.
\end{proof}

\begin{lemma}\label{when:3}
If $K$ is a compact space of weight $\omega_1$ then $|C(\omega^\ast,K)|\ge 2^{\omega_1}$.
\end{lemma}

\begin{proof}
Recall first that if $f$ is a continuous function mapping $\omega^\ast$ onto a metric space $L$ then for every $y\in L$
the set $f^{-1}(y)$ is a $G_\delta$ subset of $\omega^\ast$ and, consequently, it has a nonempty interior.

As $K$ is of weight $\omega_1$, we can express $K$  as the limit of a continuous inverse system $\langle K_\alpha: \alpha<\omega_1\rangle$
of metric compacta; for $\beta<\alpha$, denote by $\pi^\alpha_\beta:K_\alpha\lra K_\beta$ bonding maps of the  system. We can assume that
for every $\alpha$ the mapping $\pi_\alpha^{\alpha+1}:K_{\alpha+1}\lra K_\alpha$ is not injective.

We construct by induction on $\alpha<\omega_1$ continuous surjections $g_\sigma:\omega^\ast\lra K_\alpha$ , where $\sigma\in 2^\alpha$, so that
$g_{\sigma\sfrown 0}\neq g_{\sigma\sfrown 1}$ and the following diagram commutes (for $i=0,1$)
$$\xymatrix{
K_\alpha &&K_{\alpha+1}\ar[ll]_{\pi^{\alpha+1}_\alpha}\\
\omega^\ast\ar_{g_\sigma}[u]  \ar_{g_{\sigma\sfrown i}}[urr]&& }$$
We start the construction with any constant function $g_\emptyset: \omega^\ast\lra K_0$. At the successor stage,  given $g_\sigma$, the existence of $g_{\sigma\sfrown 0}$ follows directly from Theorem \ref{prel:2.5}.
To define $g_{\sigma\sfrown 1}\neq g_{\sigma\sfrown 0}$ take $t\in K_\alpha$ and distinct $s_0, s_1\in K_{\alpha+1}$ such that
$\pi^{\alpha+1}_\alpha(s_i)=t$. Then, by the remark above, $g_{\sigma\sfrown 0}=s_0$ on some
nonempty open set $V\subset \omega^\ast$. Take a nonempty clopen set $U$ properly contained in  $V$ and declare that
$g_{\sigma\sfrown 1}$ equals $s_1$ on the set $V$ and $g_{\sigma\sfrown 1}=g_{\sigma\sfrown 0}$ outside $V$.
Then the functions $g_\sigma$ are as required.

At the limit stage $\alpha<\omega_1$, by the continuity of the inverse system $\langle K_\alpha: \alpha<\omega_1\rangle$,  for $\sigma\in 2^\alpha$, there is a unique continuous  function $g_\sigma:\omega^\ast\lra K_\alpha$ such that the following diagram commutes, for every $\beta<\alpha$,
 $$\xymatrix{
K_\beta &&K_\alpha\ar[ll]_{\pi_\beta^\alpha}\\
\omega^\ast\ar_{g_{\sigma|\beta}}[u]  \ar_{g_{\sigma}}[urr]&& }$$

Now, take continuous mappings $\pi_\alpha:K\lra K_\alpha$ resulting from the inverse system and note that, by our construction,
for every $\tau\in 2^{\omega_1}$ there exists a unique continuous  function $g_\tau:\omega^\ast\lra K$ such that, for every $\alpha<\omega_1$, we have a commuting diagram
 $$\xymatrix{
K_\alpha &&K\ar[ll]_{\pi_\alpha}\\
\omega^\ast\ar_{g_{\tau|\alpha}}[u]  \ar_{g_{\tau}}[urr]&& }$$
In particular, $g_\tau\neq g_{\tau'}$ whenever $\tau\neq\tau'$, and the proof is complete.
\end{proof}

\begin{remark}
Under CH Lemma \ref{when:3} asserts that $|C(\omega^\ast,K)| = 2^\con$ whenever the weight of $K$ equals $\con$.
Let us note that this is not provable in the usual set theory.

In the Cohen model (after adding $\omega_2$ Cohen reals to a model of GCH) $\con=\omega_2$ and $2^{\omega_1}=\con$. Moreover, in the space $\omega^\ast$ any  strictly increasing
sequence of clopen sets is of length at most $\omega_1$; this result is a consequence of  Kunen's theorem
stating that in that model the set $\{(\alpha,\beta): \alpha<\beta<\omega_2\}$ is not in the $\sigma$-algebra
of subsets of $\omega_2\times\omega_2$ generated by all rectangles, see Lemma 5.3 from \cite{DH01}  and the remarks following it.

Consequently, in the Cohen model if we take the space $K=[0,\omega_2]$ (of ordinal numbers $\le\omega_2$ equipped with the order topology) then
$w(K)=\con$ and $|C(\omega^\ast,K)|=\con$. Indeed,
for a continuous function $g:\omega^\ast\lra [0,\omega_2]$ the image $I(g)=g(\omega^\ast)$ must be of size $\le\omega_1$ since, otherwise,
the chain of clopens $\{g^{-1}[0,\xi]:\xi\in I(g)\}$ would be strictly increasing.
Hence $g$ is uniquely determined by $I(g)$ and
the chain $\{g^{-1}[0,\xi]:\xi\in I(g)\}$. As  $2^{\omega_1}=\con$, this gives $|C(\omega^\ast,K)|=\con$.
\end{remark}

However, we do not know the answer to the following

\begin{problem}\label{many cont maps}
	Is $|C(\omega^\ast,K)|\ge 2^{\omega_1}$ for any nonmetrizable compact space $K$?
\end{problem}

Note that is known that we cannot prove in ZFC that every nonmetrizable compact space $K$ contains a closed subset $L$ of weight $\omega_1$, see Remark \ref{nonreflex}.
\medskip

Using Lemma \ref{when:3} we can formulate the following cardinal test for the existence of nontrivial twisted sums.

\begin{corollary}\label{when:4}
	If $Y$ is a Banach space of density $\omega_1$ and  $|Y^\ast|<2^{\omega_1}$  then $\ext(Y,c_0)\neq 0$.
\end{corollary}

\begin{proof}
Since $Y$ has density $\omega_1$, $(B_{Y^\ast}, weak^\ast)$ has weight $\omega_1$ and
$|C(\omega^\ast, B_{Y^\ast})|\ge 2^{\omega_1}$ by Lemma \ref{when:3}. Hence,
 $\ext(Y,c_0)\neq 0$ by Corollary \ref{when:2}{\em (\ref{2})}.
\end{proof}

If $Y$ is an Asplund space, then $|Y|=|Y^\ast| = dens(Y)^\omega$, so

\begin{corollary}\label{when:asplund}
	If $\mathfrak{c}<2^{\omega_1}$, then $\ext(Y,c_0)\neq 0$ for any nonseparable Asplund space $Y$ of density $\omega_1$.
\end{corollary}

\begin{corollary}\label{ADsupplement}
	If $\ext(Y,c_0)=0$, $Y^\ast$ contains a copy $K$ of $\mathbb{A}(\kappa)$ for some $\kappa$, and $\cA$ is an almost disjoint family of subsets of $\omega$ of size $\kappa$,
	then there is $L$ with $K\sub L\sub Y^\ast$ such that $L$ is homeomorphic to the Aleksandrov-Urysohn space
	associated to $\cA$.
\end{corollary}

\begin{proof}
	Simply $\bA\bU(\cA)$ is homeomorphic to a countable discrete extension of $K$ so we may apply Theorem \ref{main}.
\end{proof}

\begin{corollary}\label{special:2}
	If $(B_{Y^\ast},weak^\ast)$ is a nonmetrizable separable Rosenthal compact space, then $\ext(Y,c_0)\neq 0$
\end{corollary}

\begin{proof}
	Since $Y$ is not separable, it is easily seen that the zero vector is a non-$G_\delta$ point of $(B_{Y^\ast},weak^\ast)$.
By a theorem of Todorcevic \cite[Theorem 9]{Tod},   $(B_{Y^\ast},w^\ast)$ contains a copy $K$ of $\bA(\con)$ whose $0$ is the only accumulation point.
 Take a maximal almost disjoint family $\mathcal{A}$ of subsets of $\omega$ and let $L=K\cup\omega$ be the Aleksandrov-Urysohn space associated to
 $\cA$. Then $L$ is not Fr\'echet-Urysohn, see Section \ref{prel:adf}. But Rosenthal compact spaces are Fr\'echet-Urysohn by the theorem of Bourgain, Fremlin and Talagrand \cite{BFT}, so $L$ cannot be embedded inside any ball of $Y^\ast$, and $\ext(Y,c_0)\neq 0$ by Corollary \ref{ADsupplement}.
 \end{proof}

Recall that, for a separable Rosenthal compact space $K$, the space $M_1(K)$ is again a separable Rosenthal compact space (cf. \cite{Mar03}), hence we immediately obtain the following

\begin{corollary}\label{special:1}
$\ext(C(K),c_0)\neq 0$ for
a nonmetrizable separable Rosenthal compact space $K$.
\end{corollary}

Let us recall that,  writing $P(K)$ for  the set of all regular probability Borel measures on $K$ of size $>1$, we have
$|P(K)|=|M_1(K)|$.

\begin{corollary}\label{some:1}
	Assume that  $\mathfrak{c}<2^{\omega_1}$. If $K$ is a  compact space then $\ext(C(K),c_0)\neq 0$ provided that $K$ contains a closed subspace $L$ of weight $\omega_1$, and $|P(K)|=\con$.
\end{corollary}

\begin{proof}
	By Lemma \ref{when:3},  for $L\sub K$ with $w(L)=\omega_1$ we have
	\[\left|C(\omega^\ast, M_1(K))\right|\ge \left|C(\omega^\ast, M_1(L))\right|\ge  \left|C(\omega^\ast, L)\right|\ge 2^{\omega_1}.\]
	On the other hand, $|M_1(K)|=\con<2^{\omega_1}$,
	so the assertion follows from Corollary \ref{when:2}{(\ref{2})}.
\end{proof}

There are several classes of compacta $K$ for which $|P(K)|=|K|^\omega$, including Rosenthal compacta, compact lines, scattered and more generally fragmentable compacta, etc.;
however, there are consistent examples of spaces, even of Corson compacta, with $|K|=\con$ and $|P(K)|=2^\con$,  see \cite{DP19} for details.

The following was already noted in \cite[Theorem 2.8(b)]{MP18}; we adapt the previous argument to our present setting.

\begin{theorem}\label{some:3}
	If $K$ is a compact space of weight $\omega_1$ and $K$ does not carry a strictly positive measure then $\ext(C(K),c_0)\neq 0$.
\end{theorem}

\begin{proof}
	Since $w(K)=\omega_1$, there is a continuous surjection $f:\omega^\ast\lra K$. Suppose that $\ext(C(K),c_0)=0$; then by Theorem \ref{main}
	$f$ can be extended to a continuous function $\wh{f}:\beta\omega\lra C(K)^\ast$ (as usual, we treat $K$ as a subset of $C(K)^\ast$). Put $\mu_n=\wh{f}(n)$ for every $n$.
	If $g\in C(K)$ is non zero then $g\circ {f}\neq 0$ so $\mu_n(g)=\wh{f}(n)(g)\neq 0$ for some $n$.
	It follows that the measures $\mu_n$ distinguish elements of $C(K)$ and therefore $\sum_n 2^{-n}|\mu_n|$ is a finite strictly positive measure on $K$.
\end{proof}

\begin{problem}\label{some:4}
	Is it true that $\ext(C(K),c_0)\neq 0$ whenever $K$ does not carry a strictly positive measure (or, if $K$ is not $ccc$)?
\end{problem}

Correa and Tausk \cite{CT16} proved that $\ext(C(K),c_0)\neq 0$  whenever $K$ is a non-$ccc$ Valdivia compact space.
This may be demonstrated  using the fact that  for a such a  space $K$ there is $K_0\sub K$, where $K_0$ is of weight $\omega_1$ and still not $ccc$,  and a retraction $r:K\lra K_0$. Then $C(K_0)$ is complemented in $C(K)$ and $C(K_0)$ admits a nontrivial twisted sum with $c_0$.
We do not know if the same holds for Valdivia compacta not carrying a strictly positive measure. The problem is  that, unlike $ccc$, the property of not supporting a measure does not have an obvious reflection at the cardinal number $\omega_1$, see \cite{MaPl} for more information.

We finish this section with some comments about the case $Y=\ell_\infty = C(\beta\omega)$. It is known that $\ext(\ell_\infty, c_0)\neq 0$, see Cabello S\'{a}nchez and  Castillo \cite{CC04} or  \cite[22.5]{SIBS}. 
As we recalled in \ref{main}$(v)$, $Ext(\ell_\infty,c_0)\neq 0$ is equivalent to the existence of an operator $\ell_\infty \lra \ell_\infty/c_0$ that cannot be lifted.
The latter was explicitly  proved in \cite{CC04}; another argument  emerges form  some ideas of 
Koszmider and  Rodr\'{i}guez-Porras \cite{KR16} who considered lifting properties of operators in a slightly different context. 
Namely, the following holds and can be proved following \cite[4.3]{KR16}.

\begin{proposition}
	There exists an operator $T:\ell_\infty\lra \ell_\infty/c_0$ which is weakly compact and has a nonseparable range.
	Such an operator cannot be lifted to $\ell_\infty$.
\end{proposition}

The fact that $Ext(\ell_\infty,c_0)\neq 0$ also follows from Corollary \ref{tsad:10} presented 
later in this paper, as $\beta\omega$ can be continuously mapped onto $[0,1]^\con$. 
It is a bit suprising  that, no matter which technique we use, checking that  $\ext(\ell_\infty,c_0)\neq 0$ does require some work.

\section{$C(K)$ spaces under the continuum hypothesis}\label{ch}

We prove here that under the continuum hypothesis  $c_0$ admits a nontrivial twisted sum with every nonseparable space of the form $C(K)$. We will need a sequence of auxiliary results.

 \begin{definition}\label{ch:1}
 Given a compact space $K$, we say that    sets  $A_1,\ldots, A_n\sub K$ are separated in $K$ if there are open sets $U_i$ with
 $A_i\sub U_i\sub K$, for every $i$, and $\bigcap_{i=1}^n U_i=\emptyset$.
 \end{definition}

\begin{proposition}[CH]\label{ch:2}
Let $K$ be a compact space of weight $\con$. Suppose that $U_1,\ldots, U_n$ are open subsets of $K$ such that
whenever $f:K\lra M$ is a continuous function into a metric space $M$ then $f(U_1)\cap\ldots f(U_n)\neq\emptyset$. Then there is a countable discrete extension $L$ of $K$ such that $U_1,\ldots, U_n$ are not separated in $L$.
\end{proposition}

\begin{proof}
The assertion is a consequence of  the following claim.
\medskip

\noindent {\sc Claim.} There is a continuous surjection $f:\omega^\ast\lra K$ such that
$\bigcap_{i=1}^n \overline{f^{-1}(U_i)}\neq\emptyset$.
\medskip

Indeed, having such a function $f$ we may use Lemma \ref{prel:2} to define a countable discrete extension $K\cup\omega$ of $K$ such that
$f$ extends to a continuous surjection $g:\beta\omega\lra K\cup\omega$. Then $K\cup\omega$ is the required space: If $V_i$ are open in $K\cup\omega$
and $U_i\sub V_i$, for every $i$, then
\[\bigcap_{i=1}^n\overline{g^{-1}(V_i)}\supset \bigcap_{i=1}^n\overline{f^{-1}(U_i)}\neq\emptyset,\]
so $\bigcap_{i=1}^n {g^{-1}(V_i)}\neq\emptyset$ as well, since $\beta\omega$ is extremally disconnected (see \cite[6.2.29]{En}), and in extremally disconnected spaces $\bigcap_{i=1}^n \overline{W_i} = \overline{\bigcap_{i=1}^n W_i}$, for open $W_i$, see \cite{Te}.
Consequently, $\bigcap_{i=1}^n {V_i}\neq\emptyset$, which shows that $U_i$ are not separated in $K\cup\omega$.
\medskip

To prove {\sc Claim} we may suppose that $K\sub [0,1]^{\omega_1}$. For $\alpha<\omega_1$ write $K_\alpha=\pi_\alpha(K)$, where
$\pi_\alpha$ is the projection onto $[0,1]^\alpha$. We also consider the projections $\pi_\alpha^\beta:[0,1]^\beta\lra [0,1]^\alpha$ for $\alpha<\beta$. We shall define an increasing function $\theta:\omega_1\lra\omega_1$ and a sequence
of coherent continuous surjections $f_\alpha:\omega^\ast\lra K_{\theta(\alpha)}$ so that, for every $\alpha$, the left part of the following diagram commutes
 $$\xymatrix{
K_{\theta(\alpha)} &&K_{\theta(\alpha+1)}\ar[ll]_{\pi_{\theta(\alpha)}^{\theta(\alpha+1)}}&&K\ar[ll]_{\pi_{\theta(\alpha+1)}}\\
&& \omega^\ast\ar^{f_{\alpha}}[ull]  \ar_{f_{\alpha+1}}[u]\ar@{.>}[urr]_{f}&&
}$$
and $f:\omega^\ast\lra K$ will be the unique mapping satisfying $\pi_{\theta(\alpha)}\circ f=f_\alpha$.

Let $\langle c^\alpha=(c^\alpha_1,\ldots, c^\alpha_n): \alpha<\omega_1\rangle$ be an enumerations of all $n$-tuples of clopen sets in $\omega^\ast$ having empty intersection.

If the construction is done below a limit cardinal $\alpha$ then we define
$\theta(\alpha)=\sup_{\beta<\alpha} \theta(\beta)$ and
then $f_\alpha$ is uniquely determined.

Suppose that $f_\alpha:\omega^\ast\lra K_{\theta(\alpha)}$ is given. Using the fact that $K_{\theta(\alpha)}$ is metrizable we describe the next step dealing with
the $n$-tuple $c^\alpha$.

By our assumption on $U_i$'s applied to $\pi_{\theta(\alpha)}$ there is $t\in K_{\theta(\alpha)}$ and $x_i\in U_i$
such that $t=\pi_{\theta(\alpha)}(x_i)$ for $i=1,\ldots, n$. Define $\theta(\alpha+1)>\theta(\alpha)$ so that
for every $i$ there is a basic open set $V_i$ determined by coordinates in $\theta(\alpha+1)$ and such that $x_i\in V_i\sub U_i$.
The set $G=\{p\in \omega^\ast: f_\alpha(p)=t\}$ is a nonempty closed $G_\delta$ subset of $\omega^\ast$, so it has nonempty interior. Hence, we may take a nonempty clopen set $c\sub G$.
We have $\bigcap_{i=1}^n c^\alpha_i=\emptyset$ and therefore there is $k\le n$ such that $a=c\sm c_k^\alpha\neq\emptyset$.

By Theorem \ref{prel:2.5} there is a continuous surjection $h:\omega^\ast\lra K_{\theta(\alpha+1)}$ such that
$\pi_{\theta(\alpha)}^{\theta(\alpha+1)}\circ h=f_\alpha$. We define $f_{\alpha+1}$ by the formula
\[f_{\alpha+1}(p)=
\begin{cases}
   \pi_{\theta(\alpha+1)}(x_k) , & \text{if } p\in a, \\
    h(p) & \text{if } p\in\omega^\ast\sm a.
  \end{cases}
\]
Clearly, $\pi_{\theta(\alpha)}^{\theta(\alpha+1)}\circ f_{\alpha+1}=f_\alpha$ and the successor step is done.

The key point is that if we consider the resulting map $f:\omega^\ast\lra K$ then, looking back on the above construction,
for any $p\in a$, the points $f(p)$ and $\pi_{\theta(\alpha+1)}(x_k)$  have the same coordinates below $\theta(\alpha+1)$ so,
by the way $\theta(\alpha+1)$ is defined, we have $f(p)\in U_k$ and therefore  $f^{-1}(U_k)\not\subset c^\alpha_k$.

The inductive construction guarantees that for any clopen sets $c_1,\ldots, c_n$, if $\bigcap_{i=1}^n c_i=\emptyset$ then $f^{-1}(U_k)\not\subset c_k$ for some $k\le n$.
This means that  $\bigcap_{i=1}^n \overline{f^{-1}(U_i)}\neq\emptyset$ (cf.\ \cite[1.5.18]{En}), and {\sc Claim}  has been proved.
\end{proof}

\begin{lemma}\label{ch:3}
Let $G$ be a subset of a compact space $K$ such that for every continuous mapping $g:K\lra Z$ into a metric space $Z$, either
$g$ is not injective on $G$ or \[g(G)\cap g(K\sm G)\neq\emptyset.\]

Then for every continuous mapping $f:M_1(K)\lra Z$ into a metric space  $Z$ and $\eps\in (0,1/2)$
we have $f(0)\in\{f(\mu)\in M_1(K): \mu^+(G)> \eps\}$.
\end{lemma}

\begin{proof}
Consider the standard embedding $M_1(K)\hookrightarrow [-1,1]^{B_{C(K)}}$.
Then a continuous mapping $f:M_1(K)\lra Z$ depends on countably many coordinates and this explains the following.
\medskip

\noindent {\sc Claim.} There is a sequence of $g_n\in B_{C(K)}$ such that for every $\mu,\nu\in M_1(K)$,
if $\mu(g_n)=\nu(g_n)$ for every $n$ then $f(\mu)=f(\nu)$.
\medskip

We apply the assumption to the diagonal map $g:K\lra [-1,1]^\omega$, $g(x)=(g_n(x))_n$. Then $g$ is either not one-to-one on $G$ or
$g(G)\cap g(K\sm G)\neq\emptyset$; in either case there are $x\in G$ and $y\in K$, $x\neq y$ such that $g(x)=g(y)$.
Take $\mu=1/2(\delta_x-\delta_y)\in M_1(K)$. Then $g_n(\mu)=0$ for every $n$ so $f(0)=f(\mu)$.
As $\mu^+(G)=1/2$, this finishes the proof.
\end{proof}

\begin{lemma}\label{ch:4}
For an open subset $V$ of a compact space $L$ the following are equivalent

\begin{enumerate}[(i)]
\item there is a continuous mapping $g:L\lra  Z$ into a metric space $Z$ such that $g$ is injective on $V$ and $g(V)\cap g(L\sm V)=\emptyset$;
\item $V$ is homeomorphic to an open subset of a compact metric space;
\item $V$ is metrizable $F_\sigma$ subset of $L$.
\end{enumerate}
\end{lemma}

\begin{proof}
$(i)\to (ii)$.  We can assume that $g:L\lra Z$ is surjective. Then, by the assumptions on $g$, for every open $U\sub V$ the set
$g(U)=Z\sm g(K\sm U)$ is open. Hence, $g|V: V\lra g(V)$ is a homeomorphism.

The implication $(ii)\to (iii)$ is obvious so it remains to prove $(iii)\to (i)$. We may represent $V$ as an increasing union $\bigcup_{n}F_n$ of metrizable compacta $F_n$. For each $n$, take a embedding $h_n: F_n\to [0,1]^\omega$ and its continuous extension $g_n: L\lra [0,1]^\omega$. Since $V$ is an open $F_\sigma$ set in a normal space $L$, we can find a continuous  function $f: L\lra [0,1]$ such that $f^{-1}((0,1]) = V$. One can easily verify that the diagonal of $f$ and all functions $g_n$ has the required properties.
\end{proof}

\begin{proposition}[CH]\label{ch:5}
Let $K$ be a compact space such that $\ext(C(K),c_0)=0$. Then there is a natural number $n$ such that
whenever $L$ is a closed subspace of $K$ of weight $\con$ and $U_1,\ldots, U_n\sub K$ are pairwise disjoint open sets then
$L\cap U_i$ is metrizable and $F_\sigma$ for some $i\le n$.
\end{proposition}

\begin{proof}
Fix $\eps\in (0,1/2)$; as $\ext(C(K),c_0)=0$, by Theorem \ref{main2} there is $n$ such that
every countable discrete extension of $M_1(K)$ can be represented in $ M_{\eps \cdot n}(K)$.

Suppose that the assertion fails for $L\sub K$ and $U_1,\ldots, U_n$. By Lemma \ref{ch:3} and Lemma \ref{ch:4}
for any continuous mapping $f: M_1(L)\lra Z$ with $Z$ metric,
writing \[V_i=\{\mu\in M_1(L):\mu^+(U_i\cap L)>\eps\},\] we have
$f(0)\in f(V_i)$ for every $i\le n$.

 Therefore $V_1,\ldots, V_n$ are open sets in $M_1(L)$ (see Lemma \ref{prel:-1}) satisfying the assumption of Proposition \ref{ch:2}, and the proposition
 says that there is a countable discrete extension  $M_1(L)\cup \omega$ in which $V_i$'s cannot be separated (note that $w(M_1(L))=\con$ since density of $C(L)=\con$).
 Then $M_1(K)\cup\omega$ is a countable discrete extension of $M_1(K)$ which should be represented in
 $ M_{\eps \cdot n} (K)$. On the other hand,
  \[ V_i\sub W_i=\{\mu\in M_{\eps \cdot n}(K): \mu^+(U_i) >\eps \} \mbox{ and } \bigcap_{i=1}^n W_i=\emptyset,\]
since $U_i$ are pairwise disjoint, and this is a contradiction.
\end{proof}

\begin{corollary}[CH]\label{ch:6}
If  $K$ is  a compact space and $\ext(C(K),c_0)=0$
then there is a finite set $F\sub K$ such that
whenever $L$ is closed subspace  of $K$ of weight $\con$
then $L\sm F$ is locally metrizable.
\end{corollary}

\begin{proof}
Let $n$ be the number given by Proposition \ref{ch:5}. Suppose that no finite set $F$ of size $\le n-1$
satisfies the assertion. Then we can inductively choose for $i\le n$ closed sets $L_i$ of weight $\con$ and distinct points $x_i\in L_i$ such that
for every $i$ and $U\ni x_i$, the set $L_i\cap U$ is not metrizable. Take pairwise disjoint open sets $U_i$ with $U_i\ni x_i$ and $L=\bigcup_{i=1}^n L_i$
to get a contradiction with Proposition \ref{ch:5} (clearly, the space $L$ has weight $\con$, cf.\ \cite[3.1.20]{En}).
\end{proof}

\begin{theorem}[CH]\label{ch:7}
If  $K$ is  a compact space of weight $\con$ then $\ext(C(K), c_0)\neq 0$.
\end{theorem}

\begin{proof}
Suppose that  $\ext(C(K), c_0)= 0$ and take a finite set $F\sub K$ as in Corollary \ref{ch:6}.
As $w(K)=\con$, there is a covering $\cU$ of $K\setminus F$ by open sets such that $|\cU|\le\con$ and $\overline{U}\sub K\sm F$ for $U\in\cU$, so that $\overline{U}$ is metrizable.

Every $\mu\in P(K)$ vanishes outside $F\cup\bigcup\cV$ for some countable $\cV\sub\cU$. Since, for each $U\in\cU$, $|M_1(\overline{U})|=\con$,
we can calculate that $|M_1(K)|=|P(K)|=\con$. On the other hand,   Corollary \ref{when:4} together with CH
say that $|M_1(K)|>\con$, a contradiction.
\end{proof}

For the last stroke we need a result due to Juh\'asz \cite{Ju93}, stating (in particular) that under CH every compact nonmetrizable space
contains a closed subspace of weight $\con$.

\begin{theorem}[CH]\label{ch:8}
If  $K$ is  a compact nonmetrizable space  then $\ext(C(K), c_0)\neq 0$.
\end{theorem}

\begin{proof}
Suppose that $\ext(C(K), c_0)= 0$; then by Theorem \ref{ch:7} the weight of $K$ is $>\con$. On the other hand, we shall prove using Proposition \ref{ch:5} and Corollary \ref{ch:6}
 that $|K|\le\con$ and this will give a contradiction, cf.\ \cite[3.1.21]{En}.

Let $F\sub K$ be a finite set as in \ref{ch:6}.

\medskip

\noindent {\sc Claim 1.} $K\sm F$ is locally metrizable.
\medskip

Otherwise, there is a closed set $L\sub K\sm F$ which is not metrizable. By \cite[Theorem 3]{Ju93} we can assume that $w(L)=\con$. By the property of $F$, the compact space $L$ is locally metrizable, hence metrizable, a  contradiction.
\medskip

\noindent {\sc Claim 2.} If $A\sub K$ and $|A|\le \con$ then $|\overline{A}|\le \con$.
\medskip

Indeed, $K\sm F$ is locally metrizable so every $x\in\overline{A}\sm F$ is a limit of a convergent sequence from $A$.
\medskip

\noindent {\sc Claim 3.} $K$ satisfies the countable chain condition.
\medskip

Otherwise, there is a pairwise disjoint family $\{V_\xi:\xi<\omega_1\}$ of open nonempty sets and we can clearly assume that
$\overline{V_\xi}\sub K\sm F$ for every $\xi<\omega_1$. Then every $\overline{V_\xi}$ is compact and metrizable so
in particular $|V_\xi|\le\con$. Take $n$ as in Proposition \ref{ch:5} and divide $\omega_1$ into pairwise disjoint uncoutable sets $T_1,\ldots, T_n$.
 Consider now $U_i=\bigcup_{\xi\in T_i} V_\xi$ for $i\le n$ and $L=\overline{\bigcup_{\xi<\omega_1} V_\xi}$.
 Then $w(L)=\con$ by Claim 2 and we get a contradiction with Proposition \ref{ch:5} since every $U_i$ is not $ccc$ and as such cannot be
 metrizable $F_\sigma$ set.
\medskip

\noindent {\sc Claim 4.} $|K|\le\con$.
\medskip

Take a maximal family $\cU$ of pairwise disjoint open sets such that $\overline{U}\sub K\sm F$ for $U\in\cU$.
Then $\cU$ is countable by Claim 3 and, for every $U\in\cU$, we have $|U|\le\con$, hence  $|\bigcup\cU|\le\con$.  It follows by Claim 2 that $|\overline{\bigcup\cU}|\le\con$ and
$K=F\cup\overline{\bigcup\cU}$ by maximality of $\cU$.
\end{proof}

\section{Twisted sums of $c_0$ and $C(K)$  for $K$ scattered}\label{scattered}

We discuss here the following instance of our main problem.

\begin{problem}\label{scattered:0}
Let $K$ be a nonmetrizable scattered compactum; is  $\ext(C(K),c_0)\neq 0$?
\end{problem}

Here is the list of already known partial  answers to Problem \ref{scattered:0}:

\begin{itemize}
\item `yes' under CH if $K^{(\omega)}=\emptyset$, Castillo \cite{Ca16};
\item `yes' under Martin's axiom if $w(K)\ge\con$ and $K^{(\omega)}=\emptyset$, Correa \cite{Co18};
\item `no' under Martin's axiom if $K$ is separable,  $w(K)<\con$ and $K^{(3)}=\emptyset$, Marciszewski and Plebanek \cite{MP18};
\item `no' under Martin's axiom if $K$ is separable, $w(K)<\con$ and $K^{(\omega)}=\emptyset$, Correa and Tausk \cite{CT18}.
\end{itemize}

By Theorem \ref{ch:8} the answer to \ref{scattered:0} is positive  under  CH.
Note that those results indicate that there are nonmetrizable compacta $K$ for which the question whether $\ext(C(K),c_0)\neq 0$
is undecidable within the usual axioms of set theory.
We shall prove, however, that the following holds in ZFC. The proof the theorem given below is preceded by two auxiliary facts.

\begin{theorem}\label{scattered:1}
If  $K$ is a scattered compact space of finite height and cardinality $\ge \mathrm{cf}(\con)$, then $\ext(C(K), c_0)\neq 0$.
\end{theorem}

\begin{proposition}\label{scattered:2}
Each scattered compact space $K$ of countable height and cardinality $\ge \mathrm{cf}(\con)$, contains a retract $M$ with $\con\ge |M| \ge \mathrm{cf}(\con)$.
\end{proposition}

\begin{proof} Consider the family of all clopen subspaces of $K$ of cardinality $\ge \mathrm{cf}(\con)$, and pick such a subspace $L$ of minimal height $\alpha$. By compactness of $L$, $\alpha$ is a successor ordinal, i.e., $\alpha=\beta+1$. The set $L^{(\beta)}$ is finite, therefore we can partition $L$ into finitely many clopen sets containing exactly one point from $L^{(\beta)}$. One of these sets must have cardinality $\ge \mathrm{cf}(\con)$, hence, without loss of generality, we can assume that $L^{(\beta)}=\{p\}$. For every $x\in L\setminus \{p\}$ fix a clopen neighborhood $U_x$ of $x$ in $L$ such that $p\notin U_x$. Clearly, the height of $U_x$ is less than $\alpha$, so, by our choice of $L$, we have $|U_x| < \mathrm{cf}(\con)$. Since every point of $L\setminus \{p\}$ has a neighborhood of size $< \mathrm{cf}(\con)$ it follows that
\begin{equation}\label{s_r1}
(\forall  A\subset L)\,  (\forall  x\in \overline{A}\setminus \{p\})\, (\exists C\subseteq A)\quad  |C| < \mathrm{cf}(\con)\mbox{ and } x\in\overline{C}\,.
\end{equation}

For any subset $A\subseteq L\setminus \{p\}$ define
\[\varphi(A) = \overline{\bigcup\{U_x: x\in A\}}\setminus \{p\}\,.\]

Observe that, for $A$ of cardinality at most $\con$, by $|U_x| < \mathrm{cf}(\con)$ we have $|\bigcup\{U_x: x\in A\}|\le \con$.  The compact space $L$ has countable height, therefore it is sequential (cf. \cite[Corollary 3.3]{DMR76}), and we can estimate the size of the closure of a subset of $L$ using $\omega_1$ many iterations of sequential closure. Hence
\begin{equation}\label{s_r2}
|\varphi(A)|\le \con, \mbox{ provided } |A|\le \con\,.
\end{equation}

Fix any subset $A\subseteq L\setminus \{p\}$ of cardinality $ \mathrm{cf}(\con)$. We define inductively, for any $\alpha< \mathrm{cf}(\con)$, sets  $A_\alpha\subseteq L\setminus \{p\}$. We start with $A_0=A$, and at successor stages we put $A_{\alpha+1}= \varphi(A_\alpha)$. If $\alpha$ is a limit ordinal we define  $A_\alpha = \bigcup\{A_\beta: \beta<\alpha\}$. Finally we take $B = \bigcup\{A_\alpha: \alpha<\mathrm{cf}(\con)\}$. From (\ref{s_r2}) we conclude that $|B|\le \con$. First, observe that $B$ is open in $L$, since, for any $x\in B$, $x$ belongs to some $A_\alpha$, and then $U_x\subseteq A_{\alpha+1}\subseteq B$. Second, the union $M= B\cup \{p\}$ is closed in $L$. Indeed, if $x\in \overline{B}\setminus \{p\}$, then by (\ref{s_r1}), there is $C\subseteq A$ with  $|C| < \mathrm{cf}(\con)$ and $x\in\overline{C}$. We have $C\subseteq A_\alpha$, for some $\alpha< \mathrm{cf}(\con)$, therefore $x\in A_{\alpha+1}\subseteq B$. Now, we can define a retraction $r: L\lra M$ by
\[ r(x)= \begin{cases} x& \mbox{ for } x\in M,\\
p& \mbox{ for } x\in L\setminus M\,.
\end{cases}\]
Then $r$ is continuous since it is continuous on closed sets $M$ and $L\setminus B$. It remains to observe that $M$ is also a retract of $K$, since $L$ is clopen and therefore a retract of $K$.
\end{proof}

\begin{lemma}\label{scattered:3} Let $\lambda$ be a regular cardinal.
Every scattered compact space $K$ of finite height and cardinality $\ge\lambda$ contains a copy of a one point compactification of a  discrete space  of cardinality $\lambda$.
\end{lemma}

\begin{proof} The case $\lambda=\omega$ is an easy consequence of the fact that each infinite scattered compact space $K$ contains a nontrivial convergent sequence. Therefore we can assume that $\lambda>\omega$.
Let $n+1$ be the height of $K$. Using the same argument as at beginning of the proof of Proposition \ref{scattered:1}, without loss of generality, we can assume that $K^{(n)}=\{p\}$ and every $x\in K\setminus \{p\}$ has a clopen neighborhood $U_x$ in $K$ of size $<\lambda$. Let $k=\max\{i: |K^{(i)}|\ge \lambda\}$. Consider
\[ A= K^{(k)}\setminus\left(\bigcup\{U_x: x \in  K^{(k+1)}\setminus \{p\}\}\cup\{p\}\right)\,.\]
Observe that by our choice of $k$, the set $A$ has cardinality $\ge\lambda$. One can easily verify that the set $A$ is discrete and $p$ is the unique accumulation point of $A$. Therefore, for any subset $B\subseteq A$ of cardinality $\lambda$, $L=B\cup \{p\}$ is a one point compactification of a discrete space of required cardinality.
\end{proof}

\begin{proof}[Proof of Theorem \ref{scattered:1}]
Let us write $\kappa= \mathrm{cf}(\con)$.
Let $K$ be a scattered compact space of finite height and cardinality $\ge \kappa$, and let
$L$ be a retract of $K$ such that $\con\ge |L| \ge \kappa$, given by Proposition \ref{scattered:2}.  By Lemma \ref{prel:0}, it is enough to check that $\ext(C(L), c_0)\neq 0$.

Take a copy $S$ in $L$ of $\bA(\kappa)$, given by
Lemma \ref{scattered:3}.
Since $L$ is scattered, every measure on $L$ is purely atomic, and so  $|M(L)|= \con$.
If we supposed that $\ext(C(L), c_0)=0$ then  Corollary  \ref{when:2}(\ref{5}) would say  that $|M(L)|\ge 2^\kappa$, that
is $2^{{\rm cf}(\con)}\le\con$, which is in  contradiction with  K\"onig's Lemma.
\end{proof}

Example \ref{scattered:4} given in the appendix  demonstrates that we cannot replace the assumption on finite height of $K$ in Lemma \ref{scattered:2} by the assumption that $K$ has countable height, cf.\ also \cite[Lemma 9.5]{MP18}.

Note, however, that Corollary \ref{when:4} yields the following.

\begin{corollary}\label{scattered:5}
	If $\mathfrak{c}<2^{\omega_1}$
then $\ext(C(K),c_0)\neq 0$ for every scattered compact space $K$  of size $\omega_1$.
\end{corollary}

\section{ Almost disjoint families} \label{ad}

We consider here almost disjoint families $\cA$ of subsets of $\omega$.
Given such a family $\cA$ and $S\sub\omega$, we write $\cA\leq S$ to denote that  $A\sub^* S$ for every $A\in\cA$.

\begin{definition}\label{ad:1}
Families $\cA_1,\ldots, \cA_n$ of subsets of $\omega$ are said to be separated if there are $S_1,\ldots, S_n\sub\omega$ such
that $\bigcap_{i=1}^n S_i=\emptyset$ and $\cA_i\leq S_i$ and every $i\le n$.
\end{definition}

\newcommand{\ad}{\mathfrak a}

\begin{definition}\label{ad:2}
For a natural number $n\ge 2$  we denote by   $\ad_n$  the minimal cardinality of an almost disjoint family $\cA$ which
can be divided into pairwise disjoint parts $\cA_1,\ldots, \cA_n$ that are not separated.
We also write $\ad_\omega=\sup_n\ad_n$.
\end{definition}

Families $\mathcal{A}$ as above exist for every $n$, see \cite{AT11}, so $\mathfrak{a}_\omega\leq \mathfrak{c}$. By the classical Luzin construction there is  an almost disjoint family $\cA$ of size $\omega_1$ such that no two uncountable subfamilies of
$\cA$ are separated, see e.g.\ \cite[3.1]{Hr14}. In particular, $\ad_2=\omega_1$ and we have
\[\omega_1=\ad_2\le \ad_3\le\ldots \le\ad_\omega\le\con.\]
One  can conclude from a result due to Avil\'es and Todorcevic  \cite[Theorem 24]{AT11} that
\[ \omega_1=\ad_2< \ad_3 <\ldots \]
is relatively consistent. Moreover,  Martin's axiom implies that $\ad_3=\con$, see \cite[Section 6]{AT11}.

\begin{lemma}\label{ad:3}
The cardinal number $\ad_\omega$ is the minimal size of an almost disjoint family $\cA$ that can be written as
a disjoint union $\cA=\bigcup_{k=1}^\infty \cA_k$ where $\cA_1,\ldots \cA_n$ are not separated for every $n$.
\end{lemma}

\begin{proof}
Divide $\omega$ into infinite disjoint parts $T_2,T_3,\ldots$. Let $\cB_k$ be an almost disjoint family of subsets of $T_k$
with $|\cB_k|=\ad_k$, and such that $\cB_k$ can be divided into disjoint nonseparable parts $\cB_{k,1},\ldots, \cB_{k,k}$. Take $\cA=\bigcup_{k=2}^\infty \cB_{k}$; then $\cA$ is almost disjoint  and it can be divided into $\cA_1=\bigcup_{k\ge 2} \cB_{k,1}$ and $\cA_n=\bigcup_{k\ge n} \cB_{k,n}$, for $n\ge2$, that are as required.
\end{proof}

We shall prove that $\ad_\omega$ is bounded from above by cardinal coefficients of some classical $\sigma$-ideals.
Let $\cN$ denotes the family of  $\lambda$-null subsets of $2^\omega$, where  $\lambda$ is  the standard product measure
and $\cM$ denotes the $\sigma$-ideal of meager sets in $2^\omega$.
We also consider the $\sigma$-ideal $\cE$ of subsets of $2^\omega$ that can be covered by a countable number of closed sets of measure zero.

Recall that if $\cI$ is a proper $\sigma$-ideal of subsets of the Cantor set $2^\omega$ then
\[  {\rm non}(\cI)=\min\{ |X|: X\notin \cI\}.\]

Cardinal coefficients of $\cE$ are discussed by Bartoszy\'nski and Shelah \cite{BS92}. Clearly,  $\cE\sub \cN\cap\cM$ so
\[ {\rm non}(\cE)\le \min\left( {\rm non}(\cN), {\rm non}(\cM)\right);\]
the strong inequality in the above formula is relatively consistent \cite{BS92}.

Recall that cardinal coefficients of the classical $\sigma$-ideals do not change if we replace $2^\omega$ by any uncountable Polish space (and $\lambda$ by any
nonatomic Borel measure on it), cf.\ \cite{Ba10} and \cite{Fr5}.

The following lemma builds on a result due to Avil\'es and Todorcevic \cite[Theorem 6]{AT11}.

\begin{lemma}\label{ad:4}
 $\ad_\omega\le {\rm non}(\cE)$.
\end{lemma}

\begin{proof}
Let us fix $n\ge 2$; we shall prove that $\ad_n\le {\rm non}(\cE)$. We may think that $\cE$ is the $\sigma$-ideal of
subsets of the space $K=n^\omega$ (where $n=\{0,1,\ldots, n-1\}$) and $\lambda$ is the standard product measure on $K$, that is
$\lambda(\{x\in K: x_k=i\})=1/n$ for every $k$ and $i<n$.

We consider the full $n$-adic tree $T=n^{<\omega}=\bigcup_{k=0}^\infty T_k$, where $T_k=\{0,1,\ldots, n-1\}^k$,
using the standard notation; in particular,
$\sigma_1\prec \sigma_2$ means that $\sigma_2$ extends $\sigma_1$,
$x|k\in T_k$ denotes the restriction of  $x\in K$, $[\sigma]=\{x\in K: x|k=\sigma\}$ for $\sigma\in T_k$, and if $\sigma = (\sigma_0,\ldots,\sigma_{n-1})$, then $\sigma^\frown i = (\sigma_0,\ldots,\sigma_{n-1},i)$.

Fix a set $X\sub K$   such that $X\notin \cE$ and $|X|={\rm non}(\cE)$.
We define for every $i<n$ a family $\cA_i$ of subsets of the  tree $T$ as follows.
For $x\in K$ we put
\[ B_i(x)=\{\sigma\in T: \sigma^\frown i\prec x\}, \mbox{ and }  \cA_i=\{B_i(x):x\in X,\ B_i(x)\text{ infinite}\}.\]
We shall check that $\cA=\bigcup_{i=0}^{n-1}\cA_i$ is an almost disjoint family on $T$  and $\cA_0,\ldots, \cA_{n-1}$ are not separated.

Clearly $B_i(x)\cap B_j(x)=\emptyset$ whenever $i\neq j$. If we take $x\neq y$ and any $i,j\le n-1$ then
there is $k$ such that $x(k)\neq y(k)$ and then $B_i(x)\cap B_j(y)$ contains only sequences of length $\le k$
so such an intersection is finite. Hence $\cA$ is an almost disjoint family.
Clearly,  $\cA_i$ are pairwise disjoint so it remains to check that
$\cA_0,\ldots, \cA_{n-1}$ are not separated.

For any set  $S\sub T$ consider the sets
\[ H^i(S)=\{x\in K: B_i(x)\sub ^* S \}, \quad  H^i_k(S)=\{x\in K: B_i(x)\sm T_k \sub S  \}.\]
Note that $H^i_k(S)$ is a closed subset of $K$ and $H^i(S)=\bigcup_k H^i_k(S)$.

Take  any sets $S_i\sub T$ such that $\cA_i\leq S_i$ for $i<n$. Then $X\sub H^i(S_i)$ and $H^i_k(S_i)\subset H^i_{k+1}(S_i)$, for every $i<n, k\in\omega$, so
\[ X\sub \bigcap_{i<n}H^i(S_i)= \bigcap_{i<n} \bigcup_{k<\omega} H^i_k(S_i)= \bigcup_{k<\omega} \bigcap_{i<n}H^i_k(S_i).\]
Since $X\notin \cE$,   there must be  $k_0$ such that the set $F=\bigcap_{i<n} H^i_{k_0}(S_i)$ has positive measure.
To finish the proof we use the following.
\medskip

\noindent {\sc Claim.} If $\lambda(F)>0$  then there is $x\in F$ and $k_1$ such that
$[x|k^\frown i]\cap F\neq\emptyset$ for all $k\ge k_1$ and all $i<n$.
\medskip

The claim  follows from the Lebesgue density theorem which, in particular says that there is $x\in F$ such that
\[\lim_{k\lra\infty} \frac{\lambda( [x|k]\cap F)]}{\lambda([x|k])}=1.\]
Then for some $k_1$ and all $k\ge k_1$ we have
\[ \lambda( [x|k]\cap F)>(1-1/n)\cdot \lambda([x|k]),\]
which implies $[x|k^\frown i]\cap F\neq\emptyset$ since $\lambda([x|k^\frown i])=(1/n)\lambda([x|k])$.
\medskip

Now if we take $x\in F$ as in {\sc Claim} then for $k\ge \max(k_0, k_1)$ we have  $x|k\in S_i$ for every $i<n$ which means
that $S_0\cap S_1\cap\ldots\cap S_{n-1}$ is nonempty, and this is what we needed to check.
\end{proof}

It is now clear that $ \ad_\omega\le \min\left( {\rm non}(\cN), {\rm non}(\cM)\right)$; in particular,
it is relatively consistent that $\ad_\omega=\omega_1<\con$.

\section{Twisted sums from almost disjoint families }\label{tsad}

Let us say that a subset $Q$ of a compact space is an Aleksandrov set if $Q$ is discrete and $\overline{Q}$ is the one pont compactification of
$Q$.
The role of almost disjoint families considered in the previous section  is connected with the following fact linking two notions of separation,
those of Definition \ref{ch:1} and Definition \ref{ad:1}.

\begin{lemma}\label{tsad:1}
If a compact space $K$ contains an Aleksandrov set $Q$ and $Q=Q_1\cup\ldots Q_n$, where $|Q_i|\ge \ad_n$ for every $i$ then there is a countable discrete extension
of $K$ in which $Q_i$ cannot be separated.
\end{lemma}

\begin{proof}
Take an almost disjoint family $\cA$ of size $\ad_n$ that is divided into disjoint pieces $\cA_1,\ldots, \cA_n$ that cannot be separated.
Let $\vf:\cA\lra Q$ be an injective map such that $\vf(A)\in Q_i$ whenever $A\in\cA_i$.
Consider the space $\bA\bU(\cA)=\omega\cup\cA\cup\{\infty\}$. We can form a countable discrete extension $K\cup\omega$  of $K$
identifying $A\in\bA\bU(\cA)$ with $\vf(A)\in K$; then $\infty$ is identified with the only cluster point of $\overline{Q}$. Take open sets $U_i\sub K\cup\omega$ such that $U_i\supset Q_i$ for every $i$.  Then $S_i=U_i\cap\omega$ satisfies $\cA_i\leq S_i$ so
$\bigcap_{i=1}^n U_i \supset \bigcap_{i=1}^n S_i\neq\emptyset$, and we are done.
\end{proof}

The next result offers a  generalization of Theorem 2.3 from \cite{CT16}.

\begin{proposition}\label{tsad:2}
Let $X$ be a Banach space and let $c_n,d_n>0$ be two sequences such that
\[ r_n=\frac{n\cdot d_n}{c_n}\lra\infty.\]

Suppose that for every $n$ there exist
$\Phi_1,\ldots,\Phi_n\subset B_{X^\ast}$ and $\Psi_1,\ldots \Psi_n \sub B_X$ such that
\begin{enumerate}[(i)]
	\item  $ \left\| \sum_{i=1}^n x_{i}\right\|\le c_n,$
	for any choice of vectors  $x_i\in\Psi_i$;
	\item for every $i$ and $x^\ast\in\Phi_i$ there is $x\in\Psi_i$ such that $x^\ast(x)>d_n$;
	\item The sets $\Phi_i$ are pairwise disjoint and $|\Phi_i|\ge\ad_n$;
	\item $\Phi=\Phi_1\cup\ldots\cup\Phi_n$ is an Aleksandrov set in $B_{X^\ast}$.
\end{enumerate}
Then there is a nontrivial twisted sum of $c_0$ and $X$.
\end{proposition}

\begin{proof}
By  $(ii)$ the $weak^\ast$ open set
\[U_i=\{x^\ast\in X^\ast: x^*(x)>d_n \mbox{ for some } x\in  \Psi_i\},\]
contains $\Phi_i$.
If $x^\ast\in U_1\cap\ldots\cap U_n$ then there are $x_i\in\Psi_i$, such that
$x^\ast(x_i)>d_n$ for $i=1,\ldots, n$. Hence
\[  \|x^\ast\|\cdot c_n \ge \|x^\ast\|\cdot\|x_1+\ldots x_n\|\ge x^\ast(x_1+\ldots + x_n) > n\cdot d_n,\]
so $\|x^\ast\|>r_n$.
In other words, the sets $\Phi_i$ are separated in $r_n\cdot B_{X^\ast}$.

By Lemma \ref{tsad:1} there exists a countable discrete extension $L_n$ of $B_{X^\ast}$ in which $\Phi_i$ are not separated.
 The conclusion is that $L_n$ cannot be realized in $r_n\cdot B_{X^\ast}$. Hence, $\ext(X,c_0)\neq 0$ by Theorem \ref{main2}.
\end{proof}

\begin{theorem}\label{tsad:3}
Suppose that, for every $n$,  a compact space $K$ contains homeomorphic  copies   $F_1,\ldots, F_n$ of
 $\bA(\ad_n)$, such that if $z_i$ is the only cluster point of $F_i$ then
 there are pairwise disjoint open sets $V_i\sub K$ with
$V_i\supseteq F_i\sm\{z_i\}$ for all $i\le n$. Then $\ext(C(K),c_0)\neq 0$.
\end{theorem}

\begin{proof}
Consider
\[ \Phi_i=\{1/2( \delta_x-\delta_{z_i} ): x\in F_i, x\neq z_i\}.\]
Clearly, for $x\in F_i\sm\{z_i\}$ we can find
an open set $U\sub V_i\setminus\{z_i\}$ such that $U\cap F_i=\{x\}$, and a norm-one function $g_x\in C(K)$ vanishing outside $U$, with $g_x(x)=1$.
We can apply Proposition \ref{tsad:2} with $\Psi_i=\{g_x:x\in F_i\sm\{z_i\}\}$  and $c_n=1, d_n=1/3$.
\end{proof}

\begin{corollary}\label{tsad:4}
If a compact space $K$ contains $n$ many pairwise disjoint copies of $\bA(\ad_n)$ for every $n$ then  $\ext(C(K),c_0)\neq 0$.
\end{corollary}

\begin{proof} This is a direct consequence of \ref{tsad:3} because pairwise disjoint closed sets can be  separated by pairwise disjoint open sets.
\end{proof}

\begin{corollary}\label{tsad:5} Let $\kappa$ be a cardinal number $\ge \ad_{\omega}$.
If $K$ is a compact space containing a copy of the Cantor cube $2^\kappa$ then  $\ext(C(K),c_0)\neq 0$. In particular, $\ext(C(K),c_0)\neq 0$, provided $2^\con$ embeds into $K$.
\end{corollary}

\begin{proof}
We can apply Theorem \ref{tsad:3} since it is easy to see that $2^{\kappa}$ contains a pairwise disjoint sequence of copies of
$\bA(\kappa)$ for every infinite $\kappa$.
\end{proof}

Note that the last part of the above corollary was first proved by Correa and Tausk \cite{CT16}.

\begin{corollary}\label{tsad:6}
It is relatively consistent that $\omega_1<\con$ and $\ext(C(2^{\omega_1}),c_0)\neq 0$.
\end{corollary}

Compare this to the fact that $Ext(C(2^{\omega_1}),c_0)=0$ under $MA_{\omega_1}$ \cite{MP18}.

\begin{remark}\label{tsad:7}
Note that in the space $K=2^\kappa$, where $\kappa>\omega$,  for any two copies $F_1,F_2$ of $\bA(\kappa)$
having the same cluster point $z$, the sets $F_1\sm\{z\}, F_2\sm\{z\}$ are not separated.
This may be checked directly or referring to the fact that
every Cantor cube  is an absolute retract for compact zerodimensional spaces.
\end{remark}

We finally discuss compact spaces carrying Radon measures of large Maharam type.
Recall  that a probability Radon measure $\mu$ on $K$ is {\em of (Maharam) type} $\kappa$ if the Banach space $L_1(\mu)$ has density $\kappa$.
A measure is homogeneous if it has the same type when restricted to any set of positive measure.

\begin{remark}\label{tsad:8}
Given  a cardinal number $\kappa$, consider the following properties of a compact space $K$.

\begin{enumerate}[(a)]
\item $K$ contains a copy of the Cantor cube $2^\kappa$.
\item $K$ can be continuously mapped onto $[0,1]^\kappa$.
\item $K$ carries a homogeneous Radon measure of type $\kappa$.
\end{enumerate}

Then the implications $(a)\to(b)\to (c)$ are true in general, for the latter see \cite[531E(d)]{Fr5};
$(b)\to (a)$ does not hold e.g.\ for $K=\beta\omega$ and $\kappa=\con$.
The implication $(c)\to (b)$ holds in ZFC only for some special $\kappa$ such as $\con^+$, see Fremlin \cite[531]{Fr5} and Plebanek
\cite{Pl02} for further information.
Recall in particular that if a compact space carries a homogeneous measure of type $\kappa$ than it carries a homogeneous measure of type $\kappa'$ for every
$\omega\le\kappa'\le\kappa$ (\cite[531E(f)]{Fr5}).
\end{remark}

We shall need the following version of the Riemann-Lebesgue lemma, see e.g.\ Talagrand \cite{Ta}, page 3.

\begin{theorem} \label{RLL}
Let $(T,\Sigma,\mu)$ be any probability measure space and let $(g_n)_n $ be a stochastically independent uniformly bounded sequence of measurable functions $T\lra\er$
such that
$\int_T g_n\;{\rm d}\mu=0$ for every $n$. Then
\[\lim_{n\to\infty} \int_T f\cdot g_n\;{\rm d}\mu=0,\]
for every bounded measurable function $f:T\lra\er$.
\end{theorem}

\begin{theorem}\label{tsad:9}
If a compact space $K$ carries a homogeneous probability measure $\mu$ of type $\ad_\omega$ then  $\ext(C(K),c_0)\neq 0$.
\end{theorem}

\begin{proof}
By the Maharam theorem the measure algebra of $\mu$ is isomorphic to the measure algebra of the usual product measure on $2^{\ad_\omega}$.
This implies that there is a family $\{B_\xi:\xi<\ad_\omega \}$ of
 $\mu$-stochastically independent Borel sets $B_\xi\sub K$ with $\mu(B_\xi)=1/2$.

 Letting for any $B\in Bor(K)$
 \begin{equation}\label{nksi}
 \nu_\xi(B)=\int_{B} h_\xi\; {\rm d}\mu, \mbox{ where } h_\xi= \chi_{B_\xi}-\chi_{K\sm B_\xi},
 \end{equation}
we get  a $weak^*$ discrete family in the unit ball of $C(K)^*$, suitably separated by elements of $C(K)$.
Namely the following holds.
\medskip

\noindent {\sc Claim 1.} Suppose that $\mu(U)=1$ for an open set $U\sub K$ so that we can assume that $B_\xi\sub U$ for every
$\xi<\kappa_0$.
Then there is a family of continuous functions $g_\xi:K\lra [-1,1]$ vanishing outside $U$
such that $\nu_\xi(g_\xi)\ge 3/4$ and $\nu_\xi(g_\eta)\le 1/4$ whenever $\xi,\eta<\kappa_0$, $\xi\neq\eta$.
\medskip

To verify the claim, for a given $B_\xi$  find closed sets $F_\xi\sub B_\xi$ and  $H_\xi\subseteq U\sm B_\xi$
with $\mu(F_\xi\cup H_\xi)>3/4$. Then we take a continuous function
$g_\xi:K\lra [-1,1]$ such that
\[
  g_\xi(x) =
  \begin{cases}
    1, & \text{for } x\in F_\xi, \\
    -1, & \text{for } x\in H_\xi, \\
     0, & \text{for } x\in K\sm U.
  \end{cases}
  \]

 Note that for the measure $\nu_\xi$ defined by \ref{nksi} and for any bounded Borel function $f$ we have
\begin{equation}\label{cov}
\nu_\xi(f)=\int_K f\; {\rm d}\nu_\xi=\int_K h_\xi\cdot f\; {\rm d}\mu.
\end{equation}

Now $h_\eta=g_\eta$ on the set $F_\eta \cup H_\eta$ of measure $\mu$ bigger than $3/4$.
  Hence,   for any $\eta$ and $\xi$ we have
\begin{equation}\label{calc}
\nu_\xi(g_\eta-h_\eta)=\int_K (g_\eta-h_\eta)\cdot h_\xi\;{\rm d}\mu\le 1/4.
\end{equation}

Taking $\eta=\xi$ this gives
 $\nu_\xi(g_\xi)\ge \nu_\xi(h_\xi)- 1/4=1-3/4=3/4$ for every $\xi$.

If $\xi\neq\eta$ then $\nu_\xi(h_\eta)=0$ by \ref{cov} and stochastic  independence of functions $h_\xi$'s.
Hence $\nu_\xi(g_\eta)=\nu_\xi(g_\eta-h_\eta)\le 1/4.$ This completes the proof of Claim.
\medskip

\noindent {\sc Claim 2.}
Let $\Phi=\{\nu_\xi:\xi<\ad_\omega\}$, where $\nu_\xi$ is defined by \ref{nksi}. Then
$\Phi\cup\{0\}$ is homeomorphic to $\bA(\ad_\omega)$.
\medskip

Indeed, it follows directly from Theorem \ref{RLL} that $0$ is the only cluster point of $\Phi$.
 \medskip

To finish  the argument note first that there is a pairwise disjoint sequence of open sets $U_n$ with $\mu(U_n)>0$.
For every $n$ we consider the measure $\mu_n$,
\[ \mu_n(B)=\frac{1}{\mu(U_n)}\cdot \mu(B\cap U_n) \mbox{ for } B\in Bor(K),\]
which is again of type $\ad_\omega$. Using the above construction we define measures $\nu_{n,\xi}$ vanishing outside $U_n$ and continuous functions
$g_{n,\xi}$ such that $g_{n,\xi}=0$ outside $U_n$. In particular $\nu_{n,\xi}(g_{k,\eta})=0$ whenever $k\neq n$.
We can now apply Proposition \ref{tsad:2} for $\Phi_i=\{\nu_{i,\xi}: \xi < \ad_\omega\}$ $\Psi_i=\{g_{i,\xi}: \xi < \ad_\omega\}$  and $c_n=1, d_n=1/2$.
\end{proof}

The next corollary gives a strengthening of Corollary \ref{tsad:5}.

\begin{corollary} Let $\kappa$ be a cardinal number $\ge \ad_{\omega}$.
If $K$ is a compact space that can be continuously mapped onto  $[0,1]^\kappa$ then there is a nontrivial twisted sum of $c_0$ with $C(K)$. In particular, $\ext(C(K),c_0)\neq 0$, provided $[0,1]^\con$ is a continuous image of $K$.
\end{corollary}

\begin{proof}\label{tsad:10}
The usual product measure $\lambda$ on $[0,1]^\kappa$ is homogeneous of type $\kappa$. If $g:K\lra [0,1]^\kappa$ is a continuous surjection
then  the space $K$ carries a homogeneous measure of type $\kappa$, see  Remark \ref{tsad:8}.
Hence the first assertion follows from Proposition \ref{tsad:9}; for  the second assertion  recall that  $\ad_\omega\le\con$.
\end{proof}


\section{A scattered compactum of height $\omega+1$}\label{ap1}

We mention here an example of a scattered space related to Problem \ref{i:2}.

\begin{example}\label{scattered:4}
There exists a separable, scattered compact space $K$ of height $\omega+1$ and cardinality $\con$ which does not contain any copy of a one point compactification of an uncountable  discrete space.
\end{example}

\begin{proof}
We will construct inductively an increasing sequence $(K_n)$ of scattered locally compact spaces such that for all $n$
\begin{enumerate}[(i)]
\item $K_n$ has height $n$ and $K_n^{(k)} = K_{n}\sm K_k$ for $k=0,1,\dots,n$;
\item $|K_1| = \omega$;
\item $K_{n}\sm K_{n-1}$ is of size $\con$ for $n=2,\ldots$;
\item every point $x$ in $K_{n+1}\sm K_{n}$ has a clopen neighborhood $U$ in $K_{n+1}$ homeomorphic to the ordinal space $\omega^n +1$ and such that $U\cap (K_{n+1}\sm K_{n}) = \{x\}$;
\item every sequence of pairwise distinct points of $K_{n}\sm K_{n-1}$ has a subsequence convergent to a point in $K_{n+1}$.
\end{enumerate}
We start the induction declaring that $K_0=\emptyset$ and $K_1$ is an infinite countable discrete space. Suppose that $n\ge1$ and we have constructed the spaces $K_0,\dots,K_n$ satisfying conditions (i-v). Let $\mathcal{A}_n$ be a maximal almost disjoint family of countable subsets of $K_{n}\sm K_{n-1}$ of size $\con$. For each $A\in \mathcal{A}_n$ we pick a point $p_A$ in such a way that all these points are distinct and they do not belong to $K_n$. We put $K_{n+1}= K_n\cup \{p_a: A\in \mathcal{A}_n\}$ and we define the topology on this set in the following way.
First, we declare that $K_n$ is an open subspace of $K_{n+1}$. Next, we define the neighborhoods of points $p_A$. To this end, fix $A\in \mathcal{A}_n$ and enumerate it as $\{x_k: k\in\omega\}$. For every $k$, take a clopen neighborhood $U_k$ of $x_k$ as in condition (iv). Observe that, by condition (i), $x_k$ in $U_k$ corresponds to the point $\omega^n$ in $\omega^n +1$. Hence, any clopen neighborhood of $x_k$ contained in $U_k$ is again homeomorphic to  $\omega^n +1$. Therefore, refining $U_k$ if necessary, we can assume that they are pairwise disjoint. Now, we define the basic neighborhoods of the point $p_A$ as the sets $V_i= \{p_A\}\cup\bigcup_{k\ge i}U_k$ for $i\in \omega$. One can easily verify that the neighborhood $V_0$ satisfies the requirements from condition (iv); clearly, the condition (i) is also satisfied. Observe that the sequence $(x_k)$ converges to the point $p_A$, therefore the condition (v) follows from the maximality of the family $\mathcal{A}_n$.

Finally, we take $L=\bigcup_{n\in \omega}K_n$, declaring that a set $U$ is open in $L$ if $U\cap K_n$ is open in $K_n$ for any $n$. $L$ is locally compact and $L^{(k)} = L\sm K_k$ for $k\in \omega$, hence $L$ is scattered of height $\omega$.
Let $K$ be the one point compactification of $L$ obtained by adding a point $p$ to $L$. Clearly, $K$ is a separable space of height $\omega+1$ and cardinality $\con$. Suppose that $K$ contains a copy $M$ of a one point compactification of an uncountable discrete space. Since all points of $K$ distinct from $p$ have countable neighborhoods, $p$ must be the only nonisolated point of $M$. Take $n$ such that $S= M\cap(K_{n}\sm K_{n-1})$ is infinite. Then on one hand any sequence of pairwise distinct points of $S$ converges in $M$ to $p$, but on the other hand, by condition (v) it has a subsequence convergent to a point in $K_{n+1}$, a contradiction.
\end{proof}

\begin{remark}\label{nonreflex}
Assuming MA, every closed subset $M$ of the above compact space $L$ is either countable or of size $\con$. Indeed, let $X$ be an uncountable subset of $L$. Then $X_n = X\cap (K_{n}\sm K_{n-1})$ is uncountable for some $n\ge 2$. Let $\mathcal{B}_n = \{A\in\mathcal{A}_n: |A\cap X_n| = \omega \}$. By the maximality of $\mathcal{A}_n$, the family $\mathcal{B}_n$ must be infinite, hence we can pick distinct $A_k \in \mathcal{B}_n$ for $k\in\omega$. Put $Y = \bigcup\{A_k\cap X_n: k\in\omega\}\subset X$ and $\mathcal{C}_n = \{A\in\mathcal{A}_n: |A\cap Y| = \omega \}$. One can easily verify that $\{A\cap Y: A\in\mathcal{C}_n\}$ is an infinite maximal almost disjoint family of subsets of a countable set $Y$ and it is well-known that MA implies that such family has size $\con$. It follows that the closure of $Y$, hence also the closure of $X$, has cardinality $\con$.
\end{remark}

\end{document}